\documentclass[11pt, leqno]{amsart}
\usepackage{graphicx}%
\usepackage{multirow}%
\usepackage{amsmath,amssymb,amsfonts,mathtools, esint}%
\usepackage{amsthm}%
\usepackage{mathrsfs}%
\usepackage[title]{appendix}%
\usepackage{xcolor}%
\usepackage{textcomp}%
\usepackage{manyfoot}%
\usepackage{booktabs}%
\usepackage{algorithm}%
\usepackage{algorithmicx}%
\usepackage{algpseudocode}%
\usepackage{listings}%
\usepackage{bbm}
\usepackage[pagebackref=true, colorlinks=true, citecolor=blue]{hyperref}
\theoremstyle{thmstyleone}%
\newtheorem{theorem}{Theorem}%
\newtheorem{remark}{Remark}%
\newtheorem{defn}{Definition}%
\newtheorem{thm}{Theorem}[section]

\newtheorem{assm}{Assumption}[section]
\newtheorem{lem}[theorem]{Lemma}
\newtheorem{cor}[theorem]{Corollary}
\numberwithin{theorem}{section} 

\newcommand{\bx}{{\bf x}}
\newcommand{\be}{{\bf e}}

\newcommand{\R}{\mathbb{R}}

\usepackage[foot]{amsaddr}
\usepackage[margin=1in]{geometry}

\newcommand{\bn}{{\bf n}}
\newcommand{\bg}{{\bf g}}
\newcommand{\sU}{{\mathscr{U}}}
\newcommand{\s}{\color{black}}
\newcommand{\cred}{\color{red}}
\newcommand{\sB}{\mathcal B}
\newcommand{\sD}{\mathscr D}
\newcommand{\sE}{\mathcal E}
\newcommand{\Tr}{\text{Tr}}
\newcommand{\bQ}{{\bf Q}}

\newcommand{\sC}{\mathcal C}

\newcommand{\sO}{\mathcal O}

\newcommand{\sS}{\mathcal S}

\newcommand{\sY}{\mathcal Y}

\newcommand{\ep}{\varepsilon}

\newcommand{\sL}{\mathscr{L}}
\newcommand{\sF}{\mathcal F}
\newcommand{\bu}{{\bf u}}
\newcommand{\bq}{{\bf q}}
\newcommand{\bP}{{\mathbb P}}
\newcommand{\bw}{{\bf w}}
\newcommand{\bW}{{\bf W}}

\newcommand{\bE}{{\mathbb E}}

\newcommand{\bD}{{\bf D}}
\newcommand{\bH}{{\bf H}}
\newcommand{\bC}{{\bf C}}

\newcommand{\bU}{{\bf U}}
\newcommand{\bL}{{\bf L}}
\newcommand{\sW}{\mathscr{W}}
\newcommand{\sV}{\mathscr{V}}

\newcommand{\bv}{{\bf v}}

\newcommand{\bfeta}{\boldsymbol{\eta}}
\newcommand{\bfpsi}{\boldsymbol{\psi}}
\newcommand{\bfphi}{\boldsymbol{\phi}}
\newcommand{\bfvarphi}{\boldsymbol{\varphi}}

\renewcommand{\tilde}{\widetilde}

\newcommand{\bftau}{\boldsymbol{\tau}}
\renewcommand{\tilde}{\widetilde}

\numberwithin{equation}{section}

\title[Stochastic FSI with Navier slip boundary condition]{A stochastic fluid-structure interaction problem with the Navier slip boundary condition}
\author[K. Tawri]{Krutika Tawri$^{1}$}

\address{\newline	$^1$ Department of Mathematics, University of California Berkeley, CA, USA.}

\email{ktawri@berkeley.edu }

\begin{document}
	
	\maketitle
	
	\begin{abstract}
	We prove the existence of martingale solutions to a stochastic fluid-structure interaction problem involving a viscous, incompressible fluid flow, modeled by the Navier-Stokes equations, through a deformable elastic tube modeled by shell/membrane equations. The fluid and the structure are nonlinearly coupled via the kinematic and dynamic coupling conditions at the fluid-structure interface. This article considers the case where the structure can have unrestricted displacement and explores the Navier-slip boundary condition imposed at the fluid-structure interface, displacement of which is not known a priori and is itself a part of the solution. The proof takes a constructive approach based on a Lie splitting scheme. The geometric nonlinearity stemming from the nonlinear coupling, the possibility of random fluid domain degeneracy, the potential jumps in the tangential components of the fluid and structure velocities at the moving interface and the low regularity of the structure velocity require the development of new techniques that lead to the existence of martingale solutions.
	\end{abstract}

\maketitle

\section{Introduction}

This paper introduces a constructive approach for investigating solutions to a complex problem describing the interaction between a deformable (purely) elastic membrane and a two-dimensional viscous, incompressible fluid flow, under the influence of multiplicative stochastic forces. The fluid flow is described by the 2D Navier-Stokes equations, while the membrane is characterized by shell equations. The fluid and the structure are fully coupled across the  moving interface through a two-way coupling that ensures continuity of the normal components of their velocities and contact forces at the interface. There has been a lot of work done in the field of deterministic FSI in the past two decades (see e.g. \cite{CDEG, G08,KT12, MC13, GH16} and the references therein), however, even though there is a lot of evidence pointing to the need for studying their stochastic perturbations, the mathematical theory of stochastic FSI or, more generally, of stochastic PDEs on randomly moving domains is completely undeveloped.

The main result of this paper is the establishment of the existence of weak martingale solutions to this highly nonlinear stochastic fluid-structure interaction problem via a Lie operator splitting scheme. The recent articles \cite{TC23, T23}, which represent the only work addressing stochastic moving boundary problems, have demonstrated the existence of weak martingale solutions to FSI problems with scalar and unrestricted structural deformations, respectively. They also consider the no-slip boundary conditions imposed at the fluid-structure interface. While this is a common assumption in the blood flow literature (see e.g. \cite{QTV00}, \cite{CTGMHR06}), the slip condition is considered to be a more realistic boundary condition in modeling near-contact dynamics, such as the closure of heart valves, as it allows for the possibility of collisions (see e.g. \cite{NP10}, \cite{MC16}). 
{\bf In this work, we provide the first existence result for a stochastic moving boundary problem involving the Navier-slip boundary condition.} It also provides, for the first time, a compactness argument, in the context of FSI involving elastic structures, by constructing test functions that are allowed to have possible jumps in the tangential direction at the fluid-structure interface, which is a key feature of the slip condition. {Our compactness result thus generalizes the existing results  while also revealing hidden regularity of the structure.}
 
The first mathematical issues that we come across are related to the facts that the fluid domain boundary is a random variable, not known a priori, which can possibly degenerate in a random fashion and that
	the incompressibility condition and the Navier-slip boundary condition lead to the dependence of the test functions on the randomly moving domains and thus require us to consider random test processes which is highly unusual for typical SPDEs on fixed domains. Due to the possibility of non-zero longitudinal structural displacement, extra care has to be taken in dealing with degenerate fluid domains i.e. when the structure touches a part of the fluid domain boundary during deformation. 
	
	First, using the Arbitrary Lagrangian Eulerian (ALE) transformations, we map the fluid equations onto a fixed domain. The use of the ALE mappings and the analysis that follows is valid for as long as there is no loss of injectivity of the ALE transformation. To deal with this injectivity condition in the stochastic case we use a cut-off function and a stopping time argument. Furthermore, via the ALE maps, additional nonlinearities appear in the weak formulation of the problem that track several geometric quantities such as the fluid-structure interface tangent and normal. 

The dependence of the test functions on the domain configurations (via the ALE maps) creates issues as we apply move to a new probability space in search of martingale solutions. Hence, we {introduce} a system that approximates the original system by augmenting it by a singular term that penalizes the divergence and the boundary behavior of the fluid velocity. However, addition of this penalty term and the low temporal regularity of the solutions create further difficulties in establishing compactness which we overcome by employing non-standard compactness arguments. In establishing tightness of the laws of the approximate solutions, we also do not have the extra regularity for the structure velocity obtained from the fluid dissipation in the no-slip case. We finally show that the solutions to the approximate systems indeed converge to a desired martingale solution of the limiting equations.

Since the stochastic forcing appears not only in the structure equations but also in the fluid equations themselves, we come across additional difficulties, which are associated with the construction of the 
	appropriate "test processes" on the approximate and limiting (time-dependent and random) fluid domains.  Namely, along with the required divergence-free property on these domains, the test functions have to satisfy appropriate boundary and measurability conditions. We construct these approximate test functions by first constructing a Carath\'eodory function that gives the definition of a test function for the limiting equations
	and then by transforming this limiting test function in a way that preserves its desired properties on the approximate domains.

\section{Problem setup}\label{sec:det_setup}
We begin describing the problem by first considering the deterministic model.

	\subsection{The deterministic model and a weak formulation}
 We consider the flow of an incompressible, viscous fluid in a two-dimensional compliant cylinder $\sO=(0,L)\times(0,1)$ with a deformable lateral boundary $\Gamma$.
	The left and the right boundary of the cylinder are the inlet  and outlet for the time-dependent fluid flow. 
	We assume ``axial symmetry'' of the data and of the flow, which allows us to  consider the flow only in the upper half of the domain, with the bottom boundary fixed and equipped with the symmetry boundary conditions. 
Assume that the time-dependent fluid domain, whose displacement is not known a priori is denoted by 	$\sO_{{\bfeta}}(t)={\bfvarphi}(t,\sO)$ whereas its deformable interface is given by $\Gamma_{\bfeta}(t)={\bfvarphi}(t,\Gamma)$. Assume that ${\bfvarphi}:\sO\rightarrow\sO_{{\bfeta}}$ is a $C^1$ diffeomorphism such that
$${\bfvarphi}|_{\Gamma_{in},\Gamma_{out},\Gamma_{b}}=id,\quad det\nabla{\bfvarphi}( t,\bx)>0,$$
where the inlet, outlet and bottom boundaries of $\sO$ are given by $\Gamma_{\text{in}}=\{0\}\times (0,1),\Gamma_{\text{out}}=\{L\}\times (0,1),\Gamma_b=(0,L)\times \{0\}$ respectively.
The displacement of the elastic structure at the top lateral boundary $\Gamma$ can be identified by $(0,L)$ will be given by ${\bfeta}(t,z)={\bfvarphi}(t,z)-(z,1)$ for $z\in (0,L)$ (see Fig. \ref{fig}). The mapping ${{\bfeta}}:[0,L]\times[0,T] \rightarrow \R^2$ such that ${\bfeta}=({\eta}_z(z,t),{\eta}_r(z,t))$ is one of the unknowns in the problem.

\begin{figure}[h]
\centering	\includegraphics[scale=0.6]{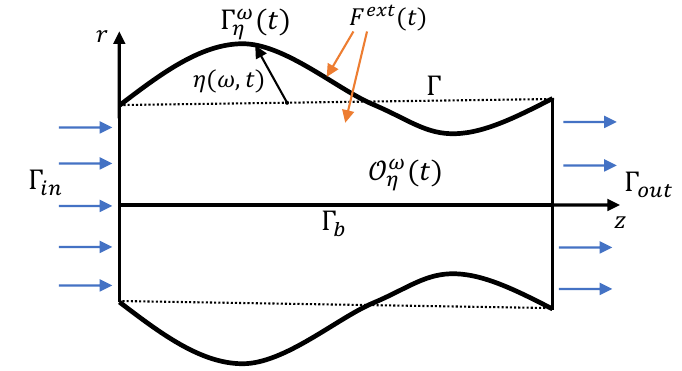}
	\caption{A realization of the fluid domain}\label{fig}
\end{figure}

\noindent{\bf The fluid subproblem:}		
The fluid flow is modeled by the incompressible Navier-Stokes equations in the 2D time-dependent domain $\sO_{{\bfeta}}(t)$:
\begin{equation}
	\begin{split}\label{u}
		\partial_t\bu + (\bu\cdot \nabla)\bu &= \nabla \cdot \sigma + F^{ext}_u\\
		\nabla \cdot \bu&=0,
	\end{split}
\end{equation}
where $\bu=(u_z,u_r)$ is the fluid velocity. The Cauchy stress tensor is $\sigma=-p I+2\nu \bD(\bu)$ where $p$ is the fluid pressure, $\nu$ is the kinematic viscosity coefficient and $\bD(\bu)=\frac12(\nabla\bu+(\nabla\bu)^T)$ is the symmetrized gradient. Here $F_u^{ext}$ represents any external forcing impacting the fluid.	{\s In this work we will be assuming that this force is random, as we shall see below.} 
The fluid flow is driven by dynamic pressure data given at the inlet and the outlet boundaries and we prescribe the symmetry boundary condition on the botton boundary as follows:
\begin{align}
		p +\frac12 |\bu|^2&=P_{in/out}(t),\qquad
		u_r=0 \quad\text{ on } \Gamma_{in/out}.\label{bc:in-out}\\
	u_r&=\partial_r u_z=0 \quad\text {on } \Gamma_b.\label{bc:bottom}
\end{align}
{\bf The structure sub-problem:} The elastodynamics problem for the displacement
${\bfeta}=(\eta_z,\eta_r)$ of the structure with respect to $\Gamma$, is as follows:
\begin{align}\label{eta}
	\partial^2_{t}{\bfeta} +\sL_e({\bfeta}) = F_{\bfeta} \quad \text{ in } (0,L),
\end{align}
where $F_{{\bfeta}}$ is the total force experienced by the structure and $\sL_e$ is a continuous, self-adjoint, coercive, linear operator on $\bH^2_0(0,L)$.	
This equation is supplemented with the following boundary conditions:
\begin{equation}\label{bc:eta}
	\begin{split}
		{\bfeta}(0)={\bfeta}(L)=\partial_z{\bfeta}(0)=\partial_z{\bfeta}(L)=0.
	\end{split}
\end{equation}
{\bf The non-linear fluid-structure coupling}: The coupling between the fluid and the structure takes place {\s across the moving fluid-structure interface. 
\begin{itemize}
\item	The kinematic coupling conditions in the {\bf Navier-slip} case are:
	\begin{align}\label{kinbc2}
		\partial_t {\bfeta}(t,z)\cdot \bn^{\bfeta}=\bu({\bfvarphi}(t,z))\cdot\bn^{\bfeta},\quad (t,z)\in[0,T]\times [0,L],\\
		(\partial_t{\bfeta}-\bu({\bfvarphi}(t,z)))\cdot\bftau^{\bfeta}=\alpha\sigma({\bfvarphi}(t,z))\bn^{\bfeta}\cdot\bftau^{\bfeta}, \quad (t,z)\in[0,T]\times [0,L].
	\end{align}
	\item
	The dynamic coupling condition is:
	\begin{align}
		F_{\bfeta}=-S_{\bfeta}(t,z) (\sigma {\bf n}^{\bfeta})|_{(t,z,{\bfvarphi}(t,z))} + F_{\bfeta}^{ext},
	\end{align}
	where $\bn^{\bfeta}(t,z)$ is the unit outward normal to the top boundary at the point, $\bftau^{\bfeta}$ is the tangent vector given by $\bftau^{\bfeta}(t,z)=\partial_z{\bfvarphi}(t,z)$ and $S_{\bfeta}(t,z)$ is the Jacobian of the transformation from Eulerian to Lagrangian coordinates. As earlier, $F^{ext}_{\bfeta}$ denotes any external force impacting the structure.
\end{itemize}
This system is supplemented with the following initial conditions:
\begin{align}\label{ic}
	\bu(t=0)=\bu_0,{\bfeta}(t=0)={\bfeta}_0, \partial_t{\bfeta}(t=0)=\bv_0.
\end{align}
\subsection*{Weak formulation on moving domain}
Using the convention that bold-faced letters denote spaces containing vector-valued functions, we define the following relevant function spaces for the fluid velocity and the structure displacement:
\begin{align*}
	&\tilde\sV_F(t)= \{{\bf u}=(u_z,u_r)\in \bH^{1}(\sO_{{\bfeta}}(t)) :
	\nabla \cdot {\bf u}=0,
	\text{ and } u_r=0 \text{ on }\partial \sO_{{\bfeta}}\setminus \Gamma_{{\bfeta}}(t) \}\nonumber\\
	&\tilde\sW_F(0,T)=L^\infty(0,T;\bL^2(\sO_{\bfeta}(\cdot))) \cap L^2(0,T;\tilde\sV_F(\cdot))\\
	&\tilde\sW_S(0,T)=W^{1,\infty}(0,T;\bL^2(0,L)) \cap L^\infty(0,T;\bH^2_0(0,L))\cap H^1(0,T;\bH^1(0,L))\\
	&\tilde\sW(0,T)=\{(\bu,{\bfeta})\in \tilde\sW_F(0,T)\times\tilde\sW_S(0,T):{ \bu({\bfvarphi}(t,z))}\cdot\bn^{\bfeta}=\partial_t{\bfeta}(t,z)\cdot\bn^{\bfeta},(t,z)\in (0,T)\times\Gamma\}.
\end{align*} 
{\s Next, we derive} a deterministic weak formulation of the problem on the moving domains. We consider $\bq \in C^1([0,T];\tilde\sV_F(\cdot))$ such that ${\bq}({\bfvarphi}(t,z))\cdot\bn^{\bfeta}=\bfpsi(t,z)\cdot\bn^{\bfeta}$ on $(0,T)\times \Gamma_{}$ for some $\bfpsi \in C^1([0,T];\bH^2_0(\Gamma))$. We multiply \eqref{u} by $\bq$, integrate in time and space and use Reynold's transport theorem to obtain,
\begin{align*}
	&(\bu(t), \bq(t))_{\sO_{{\bfeta}}(t)}= (\bu(0), \bq(0))_{\sO_{{\bfeta}}(0)} +\int_0^t\int_{\sO_{{\bfeta}}(s)} \bu(s)\cdot\partial_s\bq(s) \\
	&+ \int_0^t\int_{\Gamma_{{\bfeta}}(s)}(\bu(s)\cdot\bq(s))(\bu(s)\cdot \bn^{\bfeta}(s))
	-\int_0^t\int_{\sO_{{\bfeta}}(s)}(\bu(s)\cdot \nabla )\bu(s) \bq(s)\\
	&- 2\nu\int_0^t \int_{\sO_{{\bfeta}}(s)}  \bD(\bu(s))\cdot\bD(\bq(s))  ds+\int_0^t\int_{\partial\sO_{{\bfeta}}(s)} (\sigma \bn^{\bfeta}(s))\cdot \bq(s)
	+\int_0^t\int_{\sO_{{\bfeta}}(s)}F^{ext}_u(s)\bq(s).
\end{align*}
Set,
$$b(t,\bu,{\bf v},{\bf w}):=\frac12\int_{\sO_{{\bfeta}}(t)}\left( (\bu\cdot\nabla){\bf v}\cdot{\bf w}-(\bu\cdot\nabla){\bf w}\cdot{\bf v}\right)  ,$$
and observe that,
\begin{align*}
	-((\bu\cdot\nabla)\bu,&\bq)_{\sO_{{\bfeta}}}=-\frac12((\bu\cdot\nabla)\bu,\bq)_{\sO_{{\bfeta}}}+\frac12((\bu\cdot\nabla)\bq,\bu)_{\sO_{{\bfeta}}} -\frac12\int_{\partial\sO_{{\bfeta}}}\bu\cdot\bq \bu\cdot \bn^{\bfeta} \\
	&=-b(s,\bu,\bu,\bq) -\frac12\int_{\Gamma_{{\bfeta}}}\bu\cdot\bq\bu\cdot \bn^{\bfeta} + \frac12\int_{\Gamma_{in}}|u_z|^2q_z-\frac12\int_{\Gamma_{out}}|u_z|^2q_z.
\end{align*}
Using the divergence-free property of fluid velocity $\bu$ and the boundary conditions $u_r=0$ on $\Gamma_{in/out}$, we have that $\partial_ru_r=-\partial_zu_z=0$ on $\Gamma_{in/out}$. Hence,
\begin{align*}
	\int_{\Gamma_{in/out}}\sigma \bn^{\bfeta} \cdot \bq=\int_{\Gamma_{in/out}}\pm p\, q_z = \int_{\Gamma_{in}} \left( P_{in}-\frac12|\bu|^2\right) q_z 
	-\int_{\Gamma_{out}} \left( P_{out}-\frac12|\bu|^2\right)q_z ,
\end{align*}
whereas, $
	\int_{\Gamma_{b}}\sigma \bn^{\bfeta} \cdot \bq =0.$
We also write $\int_{\Gamma_{\bfeta}}\sigma\bn^{\bfeta}\cdot\bq$ as,
$$ \int_{\Gamma_{\bfeta}}\sigma\bn^{\bfeta}\cdot ((\bq\cdot\bn^{\bfeta})\bn^{\bfeta}+(\bq\cdot\bftau^{\bfeta})\bftau^{\bfeta})=\int_{\Gamma_{\bfeta}}\sigma\bn^{\bfeta}\cdot\bn^{\bfeta}(\bfpsi\cdot\bn^{\bfeta})+\frac1{\alpha}(\partial_t{\bfeta}-\bu)\cdot\bftau^{\bfeta}(\bq\cdot\bftau^{\bfeta}).$$

Next we multiply the structure equation \eqref{eta} by $\bfpsi$ and integrate in time and space to obtain
\begin{align*}
	(\partial_t{\bfeta}(t),\bfpsi(t))&=(\bv_0,\bfpsi(0)) +\int_0^t\int_0^L\partial_s{\bfeta}\cdot\partial_s\bfpsi dzds -\int_0^{t }\langle \sL_e(\bfeta), \bfpsi  \rangle ds\\
	&  
	-\int_0^t\int_0^LS_{\bfeta}\sigma \bn^{\bfeta} \cdot\bfpsi dzds +\int_0^t \int_0^LF_{\bfeta}^{ext}\cdot\bfpsi dz ds.
\end{align*}
Hence, in conclusion, we look for $(\bu,{\bfeta}) \in \tilde\sW(0,T)$,  that satisfies the following equation for {almost} every $t \in [0,T]$ and for any test function $\bQ=(\bq,\bfpsi)$ described above:
\begin{equation}
	\begin{split}\label{origweakform}
		&{\int_{\sO_{{\bfeta}(t )}}\bu(t )\bq(t ) d\bx+\int_0^L\partial_t{\bfeta}(t)\bfpsi(t )dz}-\int_0^{t }\int_{\sO_{{\bfeta}(s)}}\bu\cdot\partial_s\bq d\bx ds 
		-\int_0^{t}\int_0^L\partial_s{\bfeta}\partial_s\bfpsi dzds\\
		&+\int_0^{t} b(s,\bu,\bu,\bq)ds -\frac12\int_0^{t}\int_{\Gamma_{\bfeta}}(\bu\cdot\bq)(\bu\cdot\bn^{\bfeta}) dSds+ 2\nu\int_0^{t} \int_{\sO_{{\bfeta}(s)}} \bD(\bu)\cdot \bD(\bq) d\bx ds\\
		&+\frac1{\alpha}\int_0^t\int_0^L S_{\bfeta}(\partial_t{\bfeta}-\bu)\cdot\bftau^{\bfeta}((\bq-\bfpsi)\cdot\bftau^{\bfeta})dzds +\int_0^{t }\langle \sL_e(\bfeta), \bfpsi  \rangle ds\\
		&=\int_{\sO_{{\bfeta}_0}}\bu_0\bq(0) d\bx+ \int_{0}^L \bv_0\bfpsi(0) dz
		+\int_0^{t}P_{{in}}\int_{0}^1q_z\Big|_{z=0}drds-\int_0^tP_{{out}}\int_{0}^1q_z\Big|_{z=1}drds \\
		& +\int_0^{t}\int_{\sO_{{\bfeta}}(s)}\bq \cdot  F_{u}^{ext}\, d\bx ds 
		+\int_0^{t}\int_0^L \bfpsi\cdot  F_{{\bfeta}}^{ext}\, dz ds.
\end{split}\end{equation}
Here $F_{u}^{ext}$ is the volumetric external force applied to the fluid and $F_{{\bfeta}}^{ext}$ is the external force applied to the deformable boundary. 
\subsection{Stochastic framework on fixed domain}
We will take $ F^{ext}_u, F^{ext}_{\bfeta}$ to be random forces. For that purpose, consider a filtered probability space $(\Omega,\sF,(\sF_t)_{t \geq 0},\bP)$ that satisfies the usual
assumptions, i.e., $\sF_0$ is complete and the filtration is right continuous, that is, $\sF_{t}=\cap_{s \geq t}\sF_s$ for all $t \geq 0$.
	\subsubsection{ALE mappings}\label{sec:ale}
To deal the geometric nonlinearity arising due to the motion of the fluid domain, we work with the Arbitrary Lagrangian-Eulerian (ALE) mappings which are a family of diffeomorphisms from the fixed domain $\sO=(0,L)\times(0,1)$ onto {the moving domain}  $
\sO_{{\bfeta}}(t)$. Notice that the presence of the stochastic forcing implies that the domains $\sO_{\bfeta}$ are themselves random and that we must define the ALE mappings {\it pathwise}. That is, for every $\omega\in \Omega$ we will consider the maps $
A_{{\bfeta}}^\omega(t):\sO \rightarrow \sO_{{\bfeta}}(t,\omega)
$ such that $A^\omega_{\bfeta}(t)=\textbf{id}+{\bfeta}(t,\omega) \text{ on } \Gamma$ and $
A^\omega_{\bfeta}(t)=\textbf{id}  \text{ on } \partial\sO\setminus\Gamma.$

The pathwise transformed gradient, symmetrized gradient and divergence under this transformation will be denoted by 
$$\nabla^{\bfeta}f=\nabla f(\nabla A_{\bfeta})^{-1},  \, \bD^{\bfeta}(\bu)=\frac12(\nabla^{\bfeta}\bu+(\nabla^{\bfeta})^T\bu),\text{ and } div^{\bfeta}f=tr(\nabla^{\bfeta}f). $$
The Jacobian of the ALE mapping is given by $
J^\omega_{\bfeta}(t)=\text{det }\nabla A_{{\bfeta}}^\omega(t).$
Using $\bw^{\bfeta}$ to denote the ALE velocity
$\bw^{\bfeta}=\frac{d}{dt}A_{\bfeta}$, we note that $
\partial_tJ_{\bfeta}=J_{\bfeta} \nabla^{\bfeta}\cdot\bw^{\bfeta} .
$
We also rewrite the advection term as follows:
$$b^\eta(\bu,\bw,\bq)=\frac12\int_{\sO}J_{{\bfeta}}\left(((\bu-\bw^{\bfeta})\cdot\nabla^{\bfeta} )\bu\cdot\bq-((\bu-\bw^{\bfeta})\cdot\nabla^{\bfeta} )\bq\cdot\bu\right). $$
We will transform \eqref{origweakform} using these ALE maps and give the definition of martingale solutions on the fixed domain $\sO$.

We begin by describing the noise. We will assume that the external forces $F_u^{ext}$ and ${\s F_{{\bfeta}}^{ext}}$ are multiplicative stochastic forces and that we can then write the combined stochastic forcing $F^{ext}$ as follows:
\begin{equation}\label{StochasticForcing}
	{\s F^{ext}} := G(\bu,{\bfeta}) {dW},
\end{equation}
where  $W$ is a $U$-valued Wiener process with respect to the filtration $(\sF_t)_{t \geq 0}$, where $U$ is a separable Hilbert space. We denote by $Q$ 
the covariance operator of $W$, which is a positive, trace class operator on $U$, 
and define $U_0:=Q^{\frac12}(U)$. Letting $\bL^2=\bL^2(\sO)\times\bL^2(0,L)$, we now give assumptions on the noise coefficient $G$:
{\s  \begin{assm}\label{G}
		 The {noise coefficient} $G$ is a function $G:\bL^2(\sO)\times \bH^2_0(0,L) \rightarrow L_2(U_0;\bL^2)$, such that for any $\frac32\leq s<2$ the following conditions hold true:
		\begin{equation}\begin{split}\label{growthG}
				&\|G(\bu,{\bfeta})\|_{L_2(U_0;\bL^2)} \leq \|{\bu}\|_{\bL^2(\sO)} + \|\bfeta\|_{\bH^2_0(0,L)} ,\\	&\|G(\bu_1,{\bfeta}_1)-G(\bu_2,{\bfeta}_2)\|_{L_2(U_0;\bL^2)} \leq
				\|{\bu_1}-{\bu_2}\|_{\bL^2(\sO)}
				+\|{\bfeta}_1-{\bfeta}_2\|_{\bL^2(0,L)}.
		\end{split}\end{equation}	
	\end{assm}
	Here $L_2(X,Y)$ denotes the space of Hilbert-Schmidt operators from Hilbert spaces $X$ to $Y$.
	\subsubsection{Definition of martingale solutions} 
	 We will now introduce the functional framework for the stochastic problem on the  fixed reference domain $\sO=(0,L) \times (0,1)$.  The following are the functional spaces for the stochastic FSI problem defined on the fixed domain $\sO$:
	\begin{align*}
		&V= \{{\bf u}=(u_z,u_r)\in {\bH}^{1}(\sO): 
	 u_r=0 \text{ on } \Gamma_{in/out/b} \},\\
		&\sW_F=L^2(\Omega;L^\infty(0,T;\bL^2(\sO))) \cap L^2(\Omega;L^2(0,T;V)),\\
		&\sW_S=L^2(\Omega;W^{1,\infty}(0,T;\bL^2(0,L)) \cap L^\infty(0,T;\bH^2_0(0,L))\cap H^1(0,T;\bH_0^1(0,L))),\\
		&\sW(0,T)=\{(\bu,{\bfeta})\in \sW_F\times \sW_S:
		\bu|_{\Gamma}\cdot\,\bn^{\bfeta}=\partial_t{\bfeta}\cdot\bn^{\bfeta},\,\,\nabla^{\bfeta} \cdot {\bu}=0\, \,\bP-a.s.\}.
	\end{align*} 
		We also define the following spaces for test functions and fluid-structure velocities:
	\begin{equation}\label{Dspace}
		\sD=\{({\bf q},\bfpsi) \in  V \times \bH^2_0(0,L), \quad\text{and}\quad
			\sU=\{(\bu,\bv)\in V\times \bL^2(0,L).
		\}.
	\end{equation}
	\begin{defn}({\bf \s Martingale solution})\label{def:martingale}
		Given compatible deterministic initial data,  $\bu_0 \in \bL^2(\sO)$, $\bv_0\in \bL^2(0,L)$ and initial structure displacement ${\bfeta}_0\in \bH_0^2(0,L)$ 
		that satisfies
			for some $\delta_1,\delta_2>0$:
			\begin{align}	 \label{etainitial}
				{\delta_1<\inf_{\sO}J_{{\bfeta}_0},\quad \text{ and }\quad \|{\bfeta}_0\|_{\bH^2_0(0,L)}<\frac1{\delta_2}},
		\end{align}
		we say that  $(\mathscr{S},\bu,{\bfeta},T^{\bfeta})$ is a  martingale solution to the system \eqref{u}-\eqref{ic} under the assumptions \eqref{growthG} if 
		\begin{enumerate}
			\item  $\mathscr{S}=(\Omega,\sF,(\sF_t)_{t\geq 0},\bP,W)$ is a stochastic basis, that is, $(\Omega,\sF,(\sF_t)_{t\geq 0},\bP)$ is a filtered probability space satisfying the usual conditions and $W$ is a $U$-valued Wiener process.
			\item $(\bu,{\bfeta})\in \sW(0,T)$. 
			\item $T^{\bfeta}$ is a $\bP$-a.s. strictly positive, $\sF_t-$stopping time.
			\item  
			$\bu$ and ${\bfeta}$ are  $(\sF_t)_{t \geq 0}-$progressively measurable.   
			\item	For every $(\sF_t)_{t \geq 0}-$adapted, essentially bounded process $\bQ:=(\bq,\bfpsi)$ with $C^1$ paths in $\sD$ such that 
			$\nabla^{\bfeta} \cdot \bq=0$, and $\bq_\Gamma\cdot\bn^{\bfeta}=\bfpsi\cdot\bn^{\bfeta}$ the equation 
		\end{enumerate}
		\begin{equation}\label{weaksol}
			\begin{split}
				&{\int_{\sO}J_{\bfeta}(t)\bu(t)\bq(t) +\int_0^L\partial_t{\bfeta}(t)\bfpsi(t)}= \int_{\sO}J_0\bu_0\bq(0)   + \int_{0}^L \bv_0\bfpsi(0)  \\
				&+\int_0^{t }\int_{\sO}J_{\bfeta}\,\bu\cdot \partial_t\bq 	-\frac12\int_0^t\int_{\sO}J_{\bfeta}(\bu\cdot\nabla^{\bfeta }\bu \cdot\bq
				- (\bu^{}-2\bw)\cdot\nabla^{\bfeta }\bq\cdot\bu ) 	\\
				&- 2\nu\int_0^{t } \int_{\sO} J_{\bfeta}\bD^{\bfeta}(\bu)\cdot \bD^{\bfeta}(\bq) -\frac1{\alpha}\int_0^t\int_\Gamma S_{\bfeta}(\partial_t{\bfeta}-\bu)\cdot\tau^{\bfeta}((\bq-\bfpsi)\cdot\tau^{\bfeta})\\
				&+\int_0^{t }\int_0^L\partial_t{\bfeta}\partial_t\bfpsi -\int_0^{t }\langle \sL_e(\bfeta), \bfpsi  \rangle  \\
				&+\int_0^{t }\left( P_{{in}}\int_{0}^1q_z\Big|_{z=0}dr-P_{{out}}\int_{0}^1q_z\Big|_{z=1}dr\right) +\int_0^{t }(\bQ,G(\bu,{\bfeta})\,dW),
		\end{split}\end{equation}
		holds $\bP$-a.s. for almost every $t \in[0,T^{\bfeta})$.
	\end{defn}
We will next present a constructive proof of the existence of martingale solutions based on the operator splitting scheme constructed in the following section.
	\section{Operator splitting scheme} \label{sec:splitscheme}
	In this section we introduce a Lie operator splitting scheme that defines a sequence of approximate solutions to \eqref{weaksol} by semi-discretizing the problem in time. Our aim is to show that up to a subsequence, approximate solutions converge in a certain sense to a martingale solution of the stochastic FSI problem. 
	
	\subsection{Definition of the splitting scheme} We discretize the problem in time and use an operator splitting to decouple the stochastic problem into two subproblems, viz the structure and the fluid subproblems.
	We denote the time step by $\Delta t=\frac{T}{N}$ and use the notation $t^n=n\Delta t$ for $n=0,1,...,N$.	
	Let $(\bu^0,{\bfeta}^0,\bv^0)=(\bu_0,{\bfeta}_0,\bv_0)$ be the initial data. Then at the $i^{th}$ time level, we update the vector $(\bu^{n+\frac{i}{2}},{\bfeta}^{n+\frac{i}{2}},\bv^{n+\frac{i}{2}})$, for $i=1,2$ and $n=0,1,2,...,N-1$, as follows.
	\subsection*{The structure sub-problem} 
	In this sub-problem we update the structure displacement and the structure velocity while keeping the fluid velocity unchanged. That is, given $({\bfeta}^n,\bv^n) \in \bH^2_0(0,L)\times \bL^2(0,L)$ we look for a pathwise solution $({\bfeta}^{n+\frac12},\bv^{n+\frac12}) \in \bH^2_0(0,L) \times \bH^2_0(0,L)$ to the following equations: For any $\bfphi \in \bL^2(0,L)$ and $\bfpsi \in \bH^2_0(0,L)$,
	\begin{equation}
		\begin{split}\label{first}
			\bu^{n+\frac12}&=\bu^n,\\
			\int_0^L({\bfeta}^{n+\frac12}-{\bfeta}^n) \bfphi dz&= (\Delta t)\int_0^L \bv^{n+\frac12}\bfphi dz,\\
			\int_0^L \left( \bv^{n+\frac12}-\bv^n\right)  \bfpsi dz &+ (\Delta t)\langle \sL_e({\bfeta}^{n+\frac12}),\bfphi \rangle +\ep{(\Delta t)}{}\int_{0}^L\partial^2_{z} \bv^{n+\frac12}\cdot\partial^2_{z}\bfpsi dz=0.
		\end{split}
	\end{equation}
	
	For each $\omega \in \Omega$ and $n$, we define the ALE map associated with the structure variable $\bfeta^n$ as the solution to:
	\begin{equation}\label{ale}
		\begin{split}
			\Delta A^\omega_{{\bfeta}^n}&=0, \quad \text{ in } \sO,\\
			A^\omega_{{\bfeta}^n}=\textbf{id}+{\bfeta}^n(\omega) \text{ on } \Gamma,& \quad \text{ and }\quad
			A^\omega_{{\bfeta}^n}=\textbf{id}  \text{ on } \partial\sO\setminus\Gamma.
		\end{split}
	\end{equation}
	Note that we have added the last term in \eqref{first}$_3$ to regularize the structure velocity. This is required for the time derivative of the Jacobian in \eqref{martingale1} to make sense and to circumvent the issues associated with the "very weak" solutions to the Poisson equation on polygonal domains with corners. We will first pass $N\to \infty$ and then $\ep \to 0$.
\begin{lem} Consider $v\in H^{2}_0(\Gamma)$ and $w \in H^2(\sO)$ that solve,
\begin{align}\label{eqn_w}
	-\Delta w=0 \text{ in } \sO,\quad
	w=v \text{ on } \Gamma,\quad
	w=0 \text{ on } \partial\sO\setminus\Gamma.
\end{align}
Then, 
\begin{align}\label{wreg}
	\|w\|_{L^2(\sO)}\leq \|v\|_{H^{-\frac12}(\Gamma)}.
\end{align}
\end{lem}
\begin{proof}
We introduce the following system:
\begin{align*}
	-\Delta g=w \text{ in } \sO,\quad
	g=0 \text{ on }  \partial\sO.
\end{align*}
Since $w,g\in H^2(\sO)$, thanks to Theorem 1.5.3.3. in \cite{G11} (compare with Remark 1.5.3.5 in \cite{G11}), the following holds,
\begin{align*}
	(w,w)=-(\Delta g,w)=-( g,\Delta w)+\int_{\Gamma} v \,\partial_n g=\int_{\Gamma} v \,\partial_n g.
\end{align*} 
Hence,  thanks to Theorem 1.5.2.1 in \cite{G11} we then obtain
\begin{align*}
	\|w\|_{L^2(\sO)}^2 \leq \|v\|_{H^{-\frac12}(\Gamma)}\|\partial_ng\|_{ H^{\frac12}(\Gamma)} \leq \|v\|_{H^{-\frac12}(\Gamma)}\|g\|_{ H^{2}(\sO)} \leq \|v\|_{H^{-\frac12}(\Gamma)}\|w\|_{ L^{2}(\sO)}.
\end{align*}
\end{proof}

\subsection*{The fluid sub-problem}   In this sub-problem we update the fluid and the structure velocities while keeping the structure displacement the same.  As notes in the introduction, there are two major difficulties associated with constructing this half of the scheme. The {\bf first difficulty} arises because the fluid domains can degenerate randomly. Hence, we introduce an "artificial structure displacement" random variable by the means of a cut-off function as follows:
For $\delta=(\delta_1,\delta_2)$, let $\Theta_\delta$ be the
step function that satisfies
$\Theta_{\delta}(x,y)=1$ if $\delta_1 < x, \text{} y<  \frac1{\delta_2}$, 
and $\Theta_{\delta}(x,y)=0$ otherwise. 
For brevity we define,  \begin{align}\label{theta}
	\theta_\delta({\bfeta}^n):=\min_{k\leq n}\,
	\Theta_{\delta}\left( \inf_{\sO}J^k, \|{\bfeta}^k\|_{\bH^s(\Gamma)}\right)\quad \text{for a fixed $s \in (\frac{3}2,2)$},\end{align}
where $J^k(\omega)=\text{det}\nabla A^\omega_{{\bfeta}^k}$ is the Jacobian of the map defined in \eqref{ale}. Note that, 	$\theta_\delta$ is a real-valued function which tracks all the structure displacements, and is equal to 1 until the step for which the structure quantities leave the desired bounds given in terms of $\delta$. Now we define the artificial structure displacement random variable as the following stopped process,
\begin{align}\label{eta*}
	{\bfeta}^{n}_*(z,\omega)={\bfeta}^{\max_{0\leq k\leq n}\theta_\delta({\bfeta}^k)k}(z,\omega)\quad \text{for every } \omega \in \Omega,\,\, z\in[0,L].
\end{align}

Observe that, for any $p>2$ and $s\geq\frac{5}2-\frac2p$, we have the following regularity result for the harmonic extension of the boundary data associated with $\bfeta^n_*$ on a square (see Section 5 in \cite{G11}):
\begin{align}\label{A2p}
	\|A^\omega_{{\bfeta}_*^n}-{\bf id}\|_{\bW^{2,p}(\sO)} \leq C\|{\bfeta}_*^n\|_{\bW^{2-\frac1{p},p}(\Gamma)} \leq C\|{\bfeta}_*^n\|_{\bH^s(\Gamma)} .
\end{align}
Then Morrey's inequality for some $p<4$, gives us a constant $C_*>0$ (see Theorem 7.26 in \cite{GT83}) such that,
\begin{align}\label{boundJ}
	\|\nabla (A^\omega_{{\bfeta}_*^n}-{\bf id})\|_{\bC^{0,\frac1p}( \bar\sO)} \leq C\|{\bfeta}_*^n\|_{\bH^s(\Gamma)} \leq \frac{C_*}{\delta_2}, \quad \frac{5}2-\frac2p\leq s<2.
\end{align}
Theorem 5.5-1 (B) of \cite{C88}, then ensures that the map $A^{\omega}_{{\bfeta}^n_*}\in \bC^{1,\frac1p}(\bar\sO)$ is injective for any $n$ if $\delta_2$ satisfies 
\begin{align}\label{delta}
	C_* < \delta_2.
\end{align}
 Hence, such $\delta_2$, the domain configurations corresponding the the artificial variables $\bfeta^n_*$ are non-degenerate and their Jacobians have a deterministic lower bound of $\delta_1$. These artificial domain configurations will be used to define the fluid sub-problem.

The {\bf second difficulty} in constructing the dependence of the fluid test functions, through the transformed divergence-free condition and the kinematic coupling condition, on the structure displacement found in the previous sub-problem. Hence to avoid dealing with random test function we supplement the weak formulation in this sub-problem by penalty terms, via the parameter $\ep>0$, that enforce the incompressibility condition and the continuity of velocities in the normal direction, only in the limit $\ep \rightarrow 0$. 
	
\noindent{\bf A penalized system on artificial domains:}  Let $\Delta_n W:=W(t^{n+1})-W(t^n)$. Then we look for $ (\bu^{n+1} , \bv^{n+1})\in \sU_{}$ that solves	
	\begin{equation}
	\begin{split}\label{second}
		&\qquad\qquad{\bfeta}^{n+1}:={\bfeta}^{n+\frac12}, \\
		&\int_{\sO}
		J^n_*\left( \bu^{n+1}-\bu^{n+{\frac12}}\right) \bq d\bx  +\frac{1}2\int_{\sO}
		\left( J_*^{n+1}-J^n_*\right)  \bu^{n+1}\cdot\bq d\bx \\
		&+\frac12(\Delta t)\int_{\sO}J_*^n((\bu^{n+1}-
		\bw_*^{{n}})\cdot\nabla^{{\bfeta}_*^n}\bu^{n+1}\cdot\bq - (\bu^{n+1}-
		\bw_*^{{n}})\cdot\nabla^{{\bfeta}_*^n}\bq\cdot\bu^{n+1})d\bx \\
		&+2\nu(\Delta t)\int_{\sO}J^n_* \bD^{{\bfeta}_*^{n}}(\bu^{n+1})\cdot \bD^{{\bfeta}_*^{n}}(\bq) d\bx\\ 
		&+ \frac{(\Delta t)}{\ep}\int_{\sO}
		\text{div}^{{\bfeta}^n_*}\bu^{n+1}\text{div}^{{\bfeta}^n_*}\bq d\bx +\frac{(\Delta t)}{\ep}\int_\Gamma (\bu^{n+1}-\bv^{n+1})\cdot\bn^n_*(\bq-\bfpsi)\cdot\bn^n_*\\
	&	+\frac{(\Delta t)}{\alpha}\int_\Gamma S^{n}(\bu^{n+1}-\bv^{n+1})(\bq -\bfpsi)dz
		+\int_0^L(\bv^{n+1}-\bv^{n+\frac12} )\bfpsi dz \\
	&	=  \int_0^{t }\left( P^n_{{in}}\int_{0}^1q_z\Big|_{z=0}dr-P^n_{{out}}\int_{0}^1q_z\Big|_{z=1}dr\right)
	+	(G(\bu^{n},{\bfeta}_*^n)\Delta_n W, \bQ)_{\bL^2},
	\end{split}
\end{equation}
		for any $ (\bq, \bfpsi)\in\sU$. 
	 	Here we set,
	$P^n_{in/out}:=\frac1{\Delta t}\int_{t^n}^{t^{n+1}}P_{in/out}\,dt	$. Moreover,
	 the random variable $\bn^n_*$ is the unit normal to $\Gamma_{\bfeta^n_*}$ and
	$$\bw^{n}_*=\frac{1}{\Delta t}(A^\omega_{{\bfeta}_*^{n+1}}-A^\omega_{{\bfeta}_*^n}),\quad J^n_*=\text{det}\nabla A^\omega_{{\bfeta}^n_*}.$$
	\begin{remark}	
Note that, to obtain a stable scheme, we update $\bfeta^n$ in the first sub-problem  using the data from the second sub-problem and not $\bfeta_*^n$. However this means that after a certain random time, we will produce solutions that are meaningless. These discrepensies will be handled by introducing a stopping time until which the limiting solutions, corresponding to the approximations constructed in Section \ref{subsec:approxsol} using ${\bfeta}^n$'s and ${\bfeta}^n_*$'s, are equal. 
\end{remark}
Now we introduce the following discrete energy and dissipation for $i=0,1$:
	\begin{equation}\label{EnDn}
	\begin{split}
		&E^{n+\frac{i}2}=\frac12\Big(\int_{\sO}J^n_*|\bu^{n+\frac{i}2}|^2 d\bx
		+\|\bv^{n+\frac{i}2}\|^2_{\bL^2(0,L)}+\|{\bfeta}^{n+\frac{i}2}\|^2_{\bH_0^2(0,L)}\Big),\\
		&D^n_1=\ep{(\Delta t)}\int_0^L |\partial_{zz}\bv^{n+\frac12}|^2,\\
		&D^{n}_2=\Delta t \int_{\sO}\left( 2\nu J^n_*|\bD(\bu^{n+1})|^2  \right) d\bx +\frac{\Delta t}{\alpha}\int_\Gamma |\bu^{n+1}-\bv^{n+1}|^2 \\
		&+\frac{(\Delta t)}{\ep}\int_{\sO}|\text{div}^{\bfeta^n_*}\bu^{n+1}|^2dx +\frac{\Delta t}{\ep}\int_\Gamma|(\bu^{n+1}-\bv^{n+1})\cdot\bn^n_*|^2.
\end{split}\end{equation}
	\begin{lem}({\bf{Existence for the structure sub-problem.}})
		Assume that ${\bfeta}^n$ and $\bv^{n}$ are $\bH^2_0(0,L)$ and $\bL^2(0,L)$ valued $\sF_{t^n}$-measurable random variables, respectively. Then there exist $\bH^2_0(0,L)$-valued $\sF_{t^n}$-measurable random variables ${\bfeta}^{n+\frac12},\bv^{n+\frac12}$ that solve \eqref{first}, and the following semidiscrete energy inequality holds:
		\begin{equation}\label{energy1}
			\begin{split}
				E^{n+\frac12} +D^n_1+C_1^n = E^{n},
			\end{split}
		\end{equation}
		where
		$$C^{n}_1:= \frac12\|\bv^{n+\frac12}-\bv^n\|_{\bL^2(0,L)}^2   +\frac12 \|{\bfeta}^{n+\frac12}-{\bfeta}^{n}\|_{\bH_0^2(0,L)}^2,$$ corresponds to numerical dissipation.
	\end{lem}
	\begin{proof}
		The proof of existence and uniqueness of measurable solutions is straight-forward and the reader is referred to \cite{KC24} for details. This allows us to write:
		$$\bv^{n+\frac12}=\frac{{\bfeta}^{n+\frac12}-{\bfeta}^n}{\Delta t}.$$
		We now take $\bfpsi=\bv^{n+\frac12}$ in \eqref{first}$_3$ and using $a(a-b)=\frac12(|a|^2-|b|^2+|a-b|^2)$,  we obtain,
		\begin{equation}\label{energy1_1}
		\begin{split}	
				&\|\bv^{n+\frac12}\|_{\bL^2(0,L)}^2 +\|\bv^{n+\frac12}-\bv^n\|_{\bL^2(0,L)}^2  +\|{\bfeta}^{n+\frac12}\|_{\bH_0^2(0,L)}^2+\|{\bfeta}^{n+\frac12}-{\bfeta}^n\|_{\bH_0^2(0,L)}^2\\ 
				&+\ep(\Delta t)\|\partial_{zz}\bv^{n+\frac12}\|^2_{\bL^2(0,L)}
				=\|\bv^n\|_{\bL^2(0,L)}^2+ \|{\bfeta}^{n}\|_{\bH_0^2(0,L)}^2.
		\end{split}
		\end{equation}
		Recalling that $\bu^n=\bu^{n+\frac12}$ and adding the relevant terms on both sides of \eqref{energy1_1} we obtain \eqref{energy1}.
	\end{proof}
	\begin{lem}({\bf{Existence for the fluid sub-problem.}})\label{existu}
		For a given $\delta=(\delta_1,\delta_2)$ satisfying \eqref{delta}, and given $\sF_{t^n}$-measurable random variables $(\bu^{n+\frac12},\bv^n)$ taking values in $\sU$ and $\bv^{n+\frac12}$ taking values in $\bH^2_0(0,L)$, there exists an $\sF_{t^{n+1}}$-measurable random variable $\bU^{n+1}=(\bu^{n+1},\bv^{n+1})$ taking values in $\sU$ that solves \eqref{second}, and the solution satisfies the following estimate:
		\begin{equation}\label{energy2}
			\begin{split}
				E^{n+1}&+D_2^{n}+C_2^{n}\leq E^{n+\frac12} +C\Delta t((P^n_{in})^2+(P^n_{out})^2) 
				+ {C}\|\Delta_nW\|_{U_0}^2
				\|G(\bu^{n},{\bfeta}_*^n)\|_{L_2(U_0;\bL^2)}^2\\
				&+|
				\left( G(\bu^{n},{\bfeta}_*^n)\Delta_n W, (\bu^{n},\bv^n)\right)_{\bL^2} |+\frac14\int_0^L|\bv^{n+\frac12}-\bv^n|^2dz,
		\end{split}\end{equation}
		where
		$$C_2^{n}:=\frac14\int_{\sO} \left( J^n_*|\bu^{n+1}-\bu^{n}|^2 \right) d\bx +\frac14\int_0^L|\bv^{n+1}-\bv^{n+\frac12}|^2 dz,$$
		is numerical dissipation, and ${\bfeta}^n_*$ is as defined in \eqref{eta*}.
	\end{lem}
	\begin{proof}
	The proof of existence and measurability of the solutions is given using Brouwer's fixed point theorem and the Kuratowski selection theorem in \cite{TC23}.

		To obtain \eqref{energy1_1}, we take $(\bq,\bfpsi)=(\bu^{n+1},\bv^{n+1})$ in \eqref{second} and use the identity $a(a-b)=\frac12(|a|^2-|b|^2+|a-b|^2)$,:
		\begin{align*}
		&\frac12\int_{\sO}   J_*^n \left(|\bu^{n+1}|^2-|\bu^{n+\frac12}|^2 + |\bu^{n+1}-\bu^{n+\frac12}|^2 \right) 
		+\frac{1}2\int_{\sO} \left( J_*^{n+1}- J_*^n\right)  |\bu^{n+1}|^2 d\bx\\
		&+2(\Delta t)\int_{\sO} J_*^{n}|\bD^{{\bfeta}_*^{n}}(\bu^{n+1})|^2 d\bx+\frac{(\Delta t)}{\alpha}\int_\Gamma S^{n}|\bu^{n+1}-\bv^{n+1}|^2\\
		&+ \frac{(\Delta t)}{\ep}\int_{\sO} |
		\text{div}^{{\bfeta}^n_*}\bu^{n+1}|^2 d\bx +\frac{(\Delta t)}{\ep}\int_{\Gamma}| (\bu^{n+1}-\bv^{n+1})\cdot\bn^n_*|^2dz \\
		&+\frac12\int_0^L|\bv^{n+1}|^2-|\bv^{n+\frac12}|^2+|\bv^{n+1}-\bv^{n+\frac12}|^2 dz \\
		&= (\Delta t)\left( P^n_{{in}}\int_{0}^1u^{n+1}_z\Big|_{z=0}dr-P^n_{{out}}\int_{0}^1u^{n+1}_z\Big|_{z=1}dr\right) \\
		&+( G(\bu^{n},{\bfeta}_*^n)\Delta_n W, (\bU^{n+1}-\bU^{n}))+
		(G(\bu^{n},{\bfeta}_*^n)\Delta_n W, \bU^{n}).
	\end{align*}
		Here we split the discrete stochastic integral into two terms. We estimate the first term by using the Cauchy-Schwarz inequality: for some $C(\delta)>0$ independent of $n$ the following holds: 
		\begin{align*}&| 
			(G(\bu^{n}, {\bfeta}_*^n)\Delta_n W, ((\bu^{n+1},\bv^{n+1})-(\bu^{n},\bv^n) ))_{\bL^2}|\leq C\|\Delta_nW\|_{U_0}^2 \|G(\bu^{n},{\bfeta}_*^n)\|_{L_2(U_0,\bL^2)}^2 \\ &+\frac14\int_{\sO}
			J_*^n	| \bu^{n+1}-\bu^{n}|^2d\bx +\frac18\int_0^L|\bv^{n+1}-\bv^{n}|^2dz\\
			&\leq {C}\|\Delta_nW\|_{U_0}^2 \|G(\bu^{n},{\bfeta}_*^n)\|_{L_2(U_0;\bL^2)}^2 +\frac14\int_{\sO}
		J_*^n	| \bu^{n+1}-\bu^{n}|^2d\bx\\
			&+ \frac14\int_0^L|\bv^{n+1}-\bv^{n+\frac12}|^2dz+\frac14\int_0^L|\bv^{n+\frac12}-\bv^{n}|^2dz.
		\end{align*}
		We treat the terms with $P_{in/out}$ similarly to obtain \eqref{energy2}.
	\end{proof}
	Next, we will obtain uniform estimates on the expectation of the kinetic and elastic energy and dissipation of the coupled problem.
	\begin{thm}{\bf (Uniform Estimates.)}\label{energythm}
		There exists a constant $C>0$ that depends on the initial data, $\delta$, $T$, and $P_{in/out}$ and is independent of $N$ and $\ep$ such that
		\begin{enumerate}
			\item $\bE\left( \max_{1\leq n \leq N}E^{n}\right)  <C$, $ \bE\left( \max_{0\leq n\leq N-1}E^{n+\frac12}\right) <C.$
			\item $\bE\sum_{n=0}^{N-1} D^n <C$.
			\item $\bE\sum_{n=0}^{N-1}\int_{\sO}J_*^n |\bu^{n+1}-\bu^{n}|^2  d\bx+\|\bv^{n+1}-\bv^{n+\frac12}\|_{\bL^2(0,L)}^2 <C.$
			\item $\bE \sum_{n=0}^{N-1} \|\bv^{n+\frac12}-\bv^n\|_{\bL^2(0,L)}^2   + \|{\bfeta}^{n+1}-{\bfeta}^{n}\|_{\bH^2(0,L)}^2  <C,$
		\end{enumerate}
		where $D^n=D^n_1+D^n_2$ (see Definitions \eqref{EnDn}).
	\end{thm}
	\begin{proof} We first add \eqref{energy1} and \eqref{energy2} to obtain:
		\begin{equation}
			\begin{split}\label{discreteenergy}
			&	E^{n+1}+D^{n} +C^n_1+C^{n}_2\leq E^{n}+ C\Delta t((P^n_{in})^2+(P^n_{out})^2) \\
				&+C\|\Delta_nW\|_{U_0}^2
				\|G(\bu^{n},{\bfeta}_*^n)\|_{L_2(U_0;\bL^2(\sO))}^2+\left|
				(G(\bu^{n},{\bfeta}_*^n)\Delta_n W, \bU^{n}) \right|.
			\end{split}
		\end{equation}
		Then for any $m\geq 1$, summing $0\leq n\leq m-1$ gives us
		\begin{equation}\label{energysum}
			\begin{split}
				&E^m+\sum_{n=0}^{m-1}D^{n} +\sum_{n=0}^{m-1}C_1^{n}+\sum_{n=0}^{m-1}C_2^{n}
				\leq E^0+ C\,\Delta t\sum_{n=0}^{m-1}\left( (P^n_{{in}})^2 +(P^n_{{out}})^2\right) \\
				&+\sum_{n=0}^{m-1}\left|
				(G(\bu^n,{\bfeta}_*^n)\Delta_n W, \bU^{n} )\right|
				+\sum_{n=0}^{m-1} \|G(\bu^{n},{\bfeta}_*^n)\|_{L_2(U_0;\bL^2)}^2\|\Delta_nW\|_{U_0}^2.
		\end{split}\end{equation}
We take supremum over $1 \leq m \leq N$ and then take expectation on both sides of \eqref{energysum}.  
		The discrete Burkholder-Davis-Gundy inequality, \eqref{eta*}
		and \eqref{boundJ} give us for some $C(\delta)>0$ that,
		\begin{align*}
			\bE&\max_{1\leq m\leq N}|\sum_{n=0}^{m-1}
			(G(\bu^{n},{\bfeta}_*^n)\Delta_n W, \bU^{n} )| \\
			&\leq {C}\bE\left[\Delta t\sum_{n=0}^{N-1}
			\|G(\bu^{n},{\bfeta}_*^n)\|^2_{L_2(U_0,\bL^2)}\left( \left\|(\sqrt{J_*^n})\bu^{n}\right\|^2_{\bL^2(\sO)}+\|\bv^n\|_{\bL^2(0,L)}^2\right) \right]^{\frac12}\\
			& \leq {C}{}\bE\left[ \max_{0\leq m\leq N} \left( \left\|(\sqrt{J_*^m})\bu^{m}\right\|^2_{\bL^2(\sO)} + \|\bfeta^m_*\|^2_{\bH^2_0(0,L)}\right) \sum_{n=0}^{N-1}
			\Delta t \left( \|\sqrt{J^n_*}\bu^{n}\|^2_{\bL^2(\sO)}+\|\bv^n\|_{\bL^2(0,L)}^2\right) \right]^{\frac12}\\
			&\leq \frac{1}2E_0+ \frac12\bE\max_{1\leq m\leq N}\left[ \left\|(\sqrt{J_*^m})\bu^{m}\right\|^2_{\bL^2(\sO)}+\|\bfeta^m\|^2_{\bH^2_0(0,L)}\right] \\
			&+ {C \Delta t}\bE\left( \sum_{n=0}^{N-1}
			\|(\sqrt{J^n_*})\bu^n\|^2_{\bL^2(\sO)}+\|\bv^n\|_{\bL^2(0,L)}^2\right) .
		\end{align*}		
We use the tower property and \eqref{growthG} for each $n=0,...,m-1$
		to obtain,
		\begin{align}
			\bE[\|G(\bu^n,{\bfeta}_*^n)\|_{L_2(U_0,\bL^2)}^2\|\Delta_nW\|_{U_0}^2]&=\bE[\bE[\|G(\bu^n,{\bfeta}_*^n)\|_{L_2(U_0,\bL^2)}^2\|\Delta_nW\|_{U_0}^2|\sF_n]]\notag\\
			&=\bE[\|G(\bu^n,{\bfeta}_*^n)\|_{L_2(U_0,\bL^2)}^2\bE[\|\Delta_nW\|_{U_0}^2\|\sF_n]]\notag\\
			&=\Delta t(\Tr \bQ)\bE\|G(\bu^n,{\bfeta}_*^n)\|^2_{L_2(U_0,\bL^2)}\notag\\
			&\leq C\Delta t(\Tr \bQ)\bE\left( \|\sqrt{J_*^n}\bu^n\|^2_{\bL^2(\sO)}+\|\bfeta^n_*\|^2_{\bH^2_0(0,L)}\right) \label{tower}.
		\end{align}
		Thus, for some $C>0$ depending on $\delta$ and on $\Tr{\bQ}$, the following holds:
		\begin{align*}
			\bE\max_{1\leq n\leq N}&\left( \|(\sqrt{J^n_*})\bu^n\|^2_{\bL^2(\sO)} +\|\bv^n\|^2_{\bL^2(0,L)}\right) \leq CE^0+ C\|P_{in/out}\|^2_{L^2(0,T)}\\
			& +C\sum_{n=1}^{N-1} \Delta t\bE \max_{1\leq m\leq n}\left(\|(\sqrt{J^m_*})\bu^m\|^2_{\bL^2(\sO)}+\|\bv^m\|^2_{\bL^2(0,L)}\right) .
		\end{align*}
		The discrete Gronwall inequality finally gives us,
		$$\bE\max_{1\leq n \leq N}\left( \|(\sqrt{J^n_*})\bu^n\|_{\bL^2(\sO)}^2+\|\bv^n\|^2_{\bL^2(0,L)}\right) \leq C e^T,$$
		where $C$ depends only on the given data and in particular $\delta$.
		
		Hence for $E^n,D^n$ defined in \eqref{EnDn} we have,
		\begin{equation}
			\begin{split}
				\bE\max_{1\leq n\leq N}E^{n}+ \bE\sum_{n=0}^{N-1}(D^n+C_1^{n}+C_2^{n}) \leq C(\delta)(E^0  + \|P_{in/out}\|^2_{L^2(0,T)}
				+ Te^T).
			\end{split}
		\end{equation}
	\end{proof}
	
	\subsection{Approximate solutions}\label{subsec:approxsol}
	We use the solutions $(\bu^{n+\frac{i}2},{\bfeta}^{n+\frac{i}2},\bv^{n+\frac{i}2})$, $i=0,1$, defined for every $N \in \mathbb{N}\setminus \{0\}$ at discrete times to define approximate solutions on the entire interval $(0,T)$. First we introduce approximate solutions that are piece-wise constant on each sub-interval $ [n\Delta t, (n+1)\Delta t)$: 
	\begin{align}\label{AppSol}
		\bu_{N}(t,\cdot)=\bu^n, \,
		{\bfeta}_{N}(t,\cdot)={\bfeta}^n, {\bfeta}_{N}^*(t,\cdot)={\bfeta}_*^{n},  \bv_{N}(t,\cdot)=\bv^n,  {\bv}^{\#}_{N}(t,\cdot)=\bv^{n+\frac12}.
	\end{align}
	These processes are adapted to the given filtration $(\sF_t)_{t \geq 0}$.	The following are their time-shifted versions which are commonly used in deterministic settings: 
	\begin{align*}
		\bu^+_{N}(t,\cdot)=\bu^{n+1},\quad {\bfeta}^+_{N}(t,\cdot)={\bfeta}^{n+1}\quad {\bv}^+_{N}(t,\cdot)={\bv}^{n+1}\quad  t \in (n\Delta t, (n+1)\Delta t].
	\end{align*}
	We also define the corresponding piece-wise linear interpolations, for $t \in [t^n,t^{n+1}]$:
	\begin{equation}
		\begin{split}\label{approxlinear}
			&\tilde\bu_{N}( t,\cdot)=\frac{t-t^n}{\Delta t} \bu^{n+1}+ \frac{t^{n+1}-t}{\Delta t} \bu^{n}, \quad \tilde \bv_{N}(t,\cdot)=\frac{t-t^n}{\Delta t} \bv^{n+1}+ \frac{t^{n+1}-t}{\Delta t} \bv^{n}\\
			& \tilde{\bfeta}_{N}( t,\cdot)=\frac{t-t^n}{\Delta t} {\bfeta}^{n+1}+ \frac{t^{n+1}-t}{\Delta t} {\bfeta}^{n}, \quad \tilde{\bfeta}^*_{N}( t,\cdot)=\frac{t-t^n}{\Delta t} {\bfeta}^{n+1}_*+ \frac{t^{n+1}-t}{\Delta t} {\bfeta}^{n}_*.
		\end{split}
	\end{equation}
	Observe that,
	\begin{align}\label{etaderi}
		\frac{\partial\tilde{\bfeta}_{N}}{\partial t}=\bv^{\#}_{N},\quad \frac{\partial\tilde{\bfeta}^*_{N}}{\partial t}
		={\sum_{n=0}^{N-1}\theta_\delta({\bfeta}^{n+1})}\bv^{\#}_{N}\chi_{(t^n,t^{n+1})}:=\bv_{N}^*
		\quad a.e. \text{ on } (0,T),  
	\end{align}
		where $\bv^{\#}_N$ is introduced in \eqref{AppSol}.	
		In the following Lemma we summarize the results that are immediate consequences of the estimates obtain in Theorem \ref{energythm}.
	\begin{lem}\label{bounds}
		Given
		$\bu_0 \in \bL^2(\sO)$, ${\bfeta}_0 \in  \bH^2_0(0,L)$, $\bv_0 \in \bL^2(0,L)$, 
		for a fixed $\delta=(\delta_1,\delta_2)$ satisfying \eqref{delta}, we have that 
		\begin{enumerate}
			\item $\{{\bfeta}_{N}\},\{{\bfeta}^+_{N}\},\{{\bfeta}_{N}^*\}$ and thus $\{\tilde {\bfeta}_{N}\},\{\tilde{\bfeta}_{N}^*\}$ are bounded independently of $N$ and $\ep$ in\\ $L^2(\Omega;L^\infty(0,T;\bH^2_0(0,L)))$.
			\item $\{\bv_{N}\},\{{\bv}^+_{N}\},\{\bv_{N}^{\#}\},\{\bv_{N}^{*}\}$ are bounded independently of $N$ and $\ep$ in\\ $L^2(\Omega;L^\infty(0,T;\bL^2(0,L)))$.
			\item $\{\bu_{N}\}$ is bounded independently of $N$ and $\ep$ in\\ $L^2(\Omega;L^\infty(0,T;\bL^2_{}(\sO_{})))\cap L^2(\Omega;L^2(0,T;\bL^2_{}(\sO_{})))$.
			\item $\{\bu^+_{N}\}$ is bounded independently of $N$ and $\ep$ in 
			$		 L^2(\Omega;L^2(0,T;\bH^1(\sO)).$
			\item $\{\frac1{\sqrt{\ep}}\text{div}^{{\bfeta}^*_{N}}\bu^+_{N}\}$ is bounded independently of $N$ and $\ep$ in $L^2(\Omega;L^2(0,T;L^{2}(\sO)))$.
			\item $\{\frac1{\sqrt{\ep}}(\bu^+_N-\bv^+_N)\cdot\bn^*_N\}$ is bounded independently of $N,\ep$ in $L^2(\Omega;L^2(0,T;L^{2}(\Gamma)))$
			\item $\{\sqrt{\ep}\bv^{\#}_N\}$ is bounded independently of $N$ and $\ep$ in $L^2(\Omega;L^2(0,T;\bH_0^{2}(0,L)))$.
		\end{enumerate} 
	\end{lem}
	\begin{proof}
	Observe that for each $\omega\in \Omega$, $\nabla\bu^{n+1}=\nabla^{{\bfeta}^{n}_*}\bu^{n+1} (\nabla A^\omega_{{\bfeta}^{n}_*})$. Thus we have, 
	\begin{align*}
		\delta_1 \bE\int_{\sO}|\nabla\bu^{n+1}|^2d\bx &\leq  \bE\int_{\sO}(J_*^{n})|\nabla\bu^{n+1}|^2d\bx=\bE\int_{\sO}(J^{n}_*)|\nabla^{{\bfeta}_*^{n}}\bu^{n+1}\cdot \nabla A^\omega_{{\bfeta}_*^{n}}|^2d\bx\\
		&\leq C(\delta)\bE\int_{\sO}(J^{n}_*)|\nabla^{{\bfeta}_*^{n}}\bu^{n+1}|^2d\bx \leq {K}C(\delta)\bE\int_{\sO}(J_*^{n})|\bD^{{\bfeta}_*^{n}}\bu^{n+1}|^2d\bx,
	\end{align*}
	where $K>0$ is the universal Korn constant that depends only on the reference domain $\sO$. This result follows from Lemma 1 in \cite{V12} and because of uniform bounds \eqref{boundJ} which imply that $\{ A^\omega_{\bfeta^*_N}(t); \,\omega\in\Omega, \,t\in [0,T] \}$ is compact in $\bW^{1,\infty}(\sO)$.
	Thus, thanks to Theorem \ref{energythm}, there exists $C>0$, independent of $N$, such that
	\begin{align}\label{uboundV}
		\bE\int_0^T\int_{\s}|\nabla\bu^{+}_{N}|^2d\bx ds =\bE\sum_{n=0}^{N-1} \Delta t\int_{\sO}|\nabla\bu^{n+1}|^2d\bx \leq C(\delta).
	\end{align}
		The proofs of the rest of the statements follow immediately from Theorem \ref{energythm}.
	\end{proof}
	\section{Passing $N\rightarrow\infty$}\label{sec:limit}
	In this section we will pass $N \to \infty$ by first establishing tightness of the laws of the approximate random variables defined in Section \ref{subsec:approxsol}.
	\subsection{ Tightness results}
Since we do not expect our candidate solutions to be differentiable in time, the tightness results, i.e. Lemmas \ref{tightuv} and \ref{tightl2} below, will rely on the following theorem; see \cite{Tem95, S87}:
\begin{lem}\label{compactLp}
	Denote	the translation in time by $h$ of a function $f$ by:
	$$T_h f(t,\cdot)=f(t-h,\cdot), \quad h\in \R.$$ Assume that $\mathcal{Y}_0\subset\mathcal{Y}\subset\mathcal{Y}_1$ are reflexive Banach spaces such that the embedding of $\mathcal{Y}_0$ in $\mathcal{Y}$ is compact. 
	Then for any $m>0$, the embedding	$$ \{\bu \in L^2(0,T;\mathcal{Y}_0):\sup_{0<h< T} \frac1{h^m}\|T_h\bu-\bu\|_{L^2(h,T;\sY_1)}<\infty\} \hookrightarrow L^2(0,T;\mathcal{Y}),$$	is compact.
\end{lem}
We will now obtain our first tightness result for the fluid and structure velocities by bypassing the need for higher moment estimates (see also \cite{GTW} in this context).
\begin{lem}\label{tightuv} The laws of $\bu^+_N$are tight in $L^2(0,T;\bH^{{\alpha}}(\sO))$, and that of  $\bv^+_N$ are tight in $ L^2(0,T;\bH^{-{\beta}}(0,L))$ for any $0\leq  \alpha<1$, $0<\beta<\frac12$. 
\end{lem}
\begin{proof} 
	The aim of this proof is to apply Lemma \ref{compactLp} by obtaining bounds for
	\begin{align}
	\notag	\int_h^{T} \|T_h\bu^+_N-\bu^+_N\|^2_{\bL^2(\sO)}&+\|T_h\bv^+_N-\bv^+_N\|^2_{\bH^{-\beta}(0,L)}\\
	&=	(\Delta t)\sum_{n=j}^N \|\bu^n-\bu^{n-j}\|^2_{\bL^2(\sO)} + \|\bv^n-\bv^{n-j}\|^2_{\bH^{-\beta}(0,L)},\label{shift}
	\end{align}
	for any $N$ and $h$. Here $1\leq j \leq N$ such that $h=j\Delta t-l$ for some  $l<\Delta t$. For simplicity we take $l=0$.
	
	To achieve this goal, we will construct appropriate test functions for equations \eqref{first} and \eqref{second} that will result into the term on the right side of \eqref{shift}. This is done by transforming $\bu^k$ and $\bv^k$ in a way that they can be used as test functions for the equations for $\bu^n$ and $\bv^n$. This has to be done carefully since $\bu^k$ and $\bu^n$ are defined on different physical domains. Observe also that
	we can not test \eqref{first}$_3$ directly with $\bv^k$, as it does not have the required $H^2_0$-bounds independent of $\ep$. Thus, we will use a space regularization of $\bv^k$ to arrive at the desired test function for \eqref{first} and \eqref{second}. 
	First, for any $\bg \in \bH^1(\sO)$, we extend $\bg$ by its boundary values at $\partial\sO$ constantly in the normal direction outside of $\sO$ and define it to be 0 elsewhere. Then, denoting the 1D bump function by $\rho$, we define its horizontally mollified version as follows:
\begin{align*}
	\bg_\sigma(z,r)=\int_{-1}^1 \bg(z+\sigma y,r)\rho(y)dy. 
\end{align*}
Now let $P_M$ denote the orthonormal projector in $L^2(\Gamma)$ onto the space $\text{span}_{1\leq i\leq M}\{\varphi_i\}$, where $\varphi_i$ satisfies $-\Delta \varphi_i=\gamma_i\varphi_i$ and $\varphi_i(z)=0$ when $z=0,L$. For any $\bv\in\bL^2(\Gamma)$ we notate $\bv_M=P_M\bv$ and denote by $\bw_M$ the harmonic extension of $\bv_M$ in $\sO$, such that $\bw_M=0$ on $\partial\sO\setminus\Gamma$ (cf. \eqref{eqn_w}). Similarly, let $\bw^k$ be the harmonic extension of $\bv^k$ in $\sO$ with 0 boundary values on $\partial\sO\setminus\Gamma$.
For the rest of the proof we fix,
$\sigma=\frac{L}{4}h^{\frac1{12}}$ and since $\gamma_M\sim \gamma_1 M^2$, we choose $M$ such that $\gamma_M^{}=c h^{-\frac1{6}}$.

Now we define the following function that will lead to a suitable test function for the fluid subproblem (see \eqref{Qn}):
\begin{align*}
&	\bu^{k,n}_{\sigma,M}:= (J^{n}_{*})^{-1}\nabla A_{\bfeta^n_*} \left(J^{k}_{*}(\nabla  A_{\bfeta^k_*})^{-1} (\bu^k-\bw^k)\right)_\sigma
	+\bw^k_M+\left( \frac{\lambda^{k}_\sigma-\lambda^{k,n}_M}{\lambda^n_0}\right) \xi_0\chi\\
	&	-\frac1{J^n_*}\sB\left( \text{div} \left( (J^{k}_{*}(\nabla  A_{\bfeta^k_*})^{-1} \bw^k)_\sigma-J^{n}_{*}(\nabla  A_{\bfeta^n_*})^{-1} \bw^k_M\right) 	-\left( \frac{\lambda^{k}_\sigma-\lambda^{k,n}_M}{\lambda^n_0}\right) \text{div}((J^n_*)^{-1}\nabla A_{\bfeta^n_*}\xi_0\chi)\right) ,
\end{align*}	
where 
\begin{itemize}
	\item $\chi(r)$ is a smooth function on $\sO$ such that $\chi(1)= 1$ and $\chi(0)= 0$ ,
	\item $\xi_0\in\bC^\infty_0(0,L)$ that satisfies $\lambda_0^{n}:=\int_0^L(\partial_z({\bf id}+ \bfeta^n_*) \times \xi_0) \neq 0$ for any $n$,	
	\item The constants $\lambda_\sigma^{k}=\int_0^L(\partial_z({\bf id}+ \bfeta^k_*) \times \bv^k)_\sigma $ and $\lambda^{k,n}_M=\int_0^L(\partial_z({\bf id}+ \bfeta^n_*) \times \bv_M^k) $, 
	\item Finally, $\sB:L^2(\sO) \to \bH^1_0(\sO)$ is Bogovski's operator (see \cite{Galdi}). $\sB$ along with the constants in the previous point are used to correct the divergence of the extra terms appearing in the definition of $\bu^{k,n}_{\sigma,M}$ due to the extension of structure velocities in the fluid domains.
\end{itemize}
Observe that to preserve 0 boundary conditions on $\partial\sO\setminus\Gamma$ we must also "squeeze" the mollified function by $\sim 1+\sigma$ as in \cite{TC23, T23}. However, we skip that for a cleaner presentation as it does not change the following estimates and argument.
Next let, $$\bv^{k,n}_{\sigma,M}:=\bv^k_M-\left( \frac{\lambda^{k}_\sigma-\lambda^{k,n}_M}{\lambda_0^n}\right) \xi_0.$$
Observe that we used Piola transformations to define the fluid test functions $\bu^{k,n}_{\sigma,M}$.
Thanks to \eqref{boundJ} and Theorem 1.7-1 in \cite{C88} we thus observe that,
\begin{align*}
&	{J^n_*}(\text{div}^{\bfeta^n_*}\bu^{k,n}_{\sigma,M})=
	\text{div}\left(J^{k}_{*}(\nabla  A_{\bfeta^k_*})^{-1} \bu^k\right)_\sigma=\left(\text{div}\left( J^{k}_{*}(\nabla  A_{\bfeta^k_*})^{-1} \bu^k\right) \right)_{\sigma}={(J^k_*\text{div}^{\bfeta^k_*}\bu^k)_{\sigma}}
\end{align*}
Hence,
\begin{equation}
	\begin{split}\label{kton}
&\|	\text{div}^{\bfeta^n_*}\bu^{k,n}_{\sigma,M}\|_{L^2(\sO)} \leq C(\delta)\|\text{div}^{\bfeta^k_*}\bu^k\|_{L^2(\sO)}, \\
&\|(\bu^{k,n}_{\sigma,M}-  \bv^{k,n}_{\sigma,M})\cdot\bn^n_*\|_{L^2(\Gamma)}\leq C(\delta)\|(\bu^{k}-  \bv^{k})\cdot\bn^k_*\|_{L^2(\Gamma)}.
\end{split}\end{equation}
Note that, to ensure that \eqref{kton}$_2$ holds, we mollified only in the tangential direction i.e. along $\Gamma_{\bfeta^k_*}$ in the definition of $\bu^{k,n}_{\sigma,M}$.

 Observe that for any $\beta<\frac12$, we have, 
\begin{align}
	\|\bv_M-\bv\|_{\bH^{-\beta}(0,L)} \leq \gamma_{M}^{-\frac{\beta}{2}}\|\bv\|_{\bL^2(0,L)},\quad 	\|\bv_\sigma-\bv\|_{\bH^{-\beta}(0,L)}  \leq \sigma^{\beta}\|\bv\|_{\bL^2(0,L)}.
\end{align}
Observe also that
	$\|T_h\bfeta^*_N-\bfeta^*_N\|_{L^\infty(0,T;\bH^s(0,L))}\leq \|\tilde\bfeta^*_N\|_{C^{0,\frac14}(0,T;\bH^s(0,L))} h^\frac14$ for any $s>\frac32$.
Thanks to these observations and \eqref{wreg}, for any $s\geq \frac32$ and $\beta<\frac12$, we calculate
\begin{align}
\notag	&|\lambda_{\sigma}^{k}-\lambda^{k,n}_M|\leq 
	\int_0^L\int_{-1}^1\left( \int^{z+\sigma y}_z|\partial_{zz}\bfeta^k_*(w)|dw\right) \,|\bv^k(z+\sigma y)| dydz\\
\notag	&+\|\bfeta^k_*\|_{\bH^s(0,L)}\|\bv^k_\sigma-\bv^k\|_{\bH^{-\frac12}(0,L)}+\|\bfeta^k_*-\bfeta^n_*\|_{\bH^s(0,L)}\|\bv^k\|_{\bL^2(0,L)} \\
\notag	&+\|\bfeta^n_*\|_{\bH^s(0,L)}\|\bv^k_M-\bv^k\|_{\bH^{-\frac12}(0,L)}\\
	&\leq C(\delta)\left( \|\bfeta^n_*\|_{\bH_0^2(0,L)}\sigma^{\frac12}+\sigma^{\beta}+{\gamma_M^{-\frac{\beta}2}}+h^{\frac14}\|\tilde\bfeta^*_N\|_{C^{0,\frac14}(0,T;\bH^s(0,L))}\right) \|\bv^k\|_{\bL^2(0,L)}.
\label{lambda}
\end{align}

		Observe that, due to the properties of mollification,
	$		\partial_{zz}\bw_\sigma=	(\partial_{zz}\bw)_\sigma $ and $		\partial_{rr}\bw_\sigma=	(\partial_{rr}\bw)_\sigma $. Hence $\bw^k_\sigma$ is harmonic with 0 boundary values on $\partial\sO\setminus \Gamma$ and such that $\bw^k_\sigma(z,1)=\bv^k_\sigma(z)$.
	
	Moreover, the operator $\sB$ can be continuously extended to a bounded operator from $W^{l,p}_0(\sO)$ to $\bW^{l+1,p}_0(\sO)$ for any $l>-2+\frac1p$ (see \cite{GHH06}, \cite{Galdi}). Combining these observations with the bounds \eqref{boundJ} and \eqref{wreg}, we obtain for $s>\frac32$, $\beta<\frac12$ and $n-j\leq k\leq n$ that
	\begin{align}
		\notag	\|\bu^{k,n}_{\sigma,M}-\bu^k&\|_{\bL^2(\sO)} \lesssim \|\nabla A_{\bfeta_*^n}-\nabla A_{\bfeta_*^k}\|_{\bL^\infty(\sO)}\|\bu^k\|_{\bL^2(\sO)} + \|\bu^k-\bu^k_{\sigma}\|_{\bL^2(\sO)}+\\ &\notag+\|\bw^k-\bw^k_{\sigma}\|_{\bL^2(\sO)} +  \|\bw^k-\bw^k_M\|_{\bL^2(\sO)} +|\lambda^{k}_{\sigma}-\lambda_M^{k,n}|\\
		&\notag	\lesssim \| {\bfeta_*^n}- {\bfeta_*^k}\|_{\bH^s(0,L)}\|\bu^k\|_{\bL^2(\sO)} + \sigma\|\bu^k\|_{\bH^1(\sO)}\\
		& \notag + \|\bv^k-\bv^k_\sigma\|_{\bH^{-1/2}(0,L)}+ \|\bv^k-\bv^k_M\|_{\bH^{-1/2}(0,L)}+ |\lambda^{k}_{\sigma}-\lambda^{k,n}_M|\\
		\notag	&\lesssim   h^{\frac14}\|\tilde\bfeta^*_N\|_{C^{0,\frac14}(0,T;\bH^s(0,L))}(\|\bu^k\|_{\bL^2(\sO)}+\|\bv^k\|_{\bL^2(0,L)} )+\sigma\|\bu^k\|_{\bH^1(\sO)}\\
		\label{u1}&+	(\sigma^{\beta}+\gamma_M^{-\frac{\beta}2}+h^{\frac14}) \|\bv^k\|_{\bL^2(0,L)}
	\end{align}
where the hidden constants depend only on $\delta$. Observe, due to \eqref{A2p}, that,
\begin{align*}
\notag	&\|\bu^{k,n}_{\sigma,M}\|_{\bH^1(\sO)} \leq \|A_{\bfeta^k_*}\|_{\bW^{2,3}(\sO)}(\|\bu^k\|_{\bH^1(\sO)} +\|\bv_\sigma^k\|_{\bH^{1}(0,L)}) + \|\bv_M^{k}\|_{\bH^1(0,L)} +|\lambda^{k}_{\sigma}-\lambda_M^{k,n}|\\
	& \leq C(\delta)(\|\bu^k\|_{\bH^1(\sO)} + {\sigma^{-1}}\|\bv^k\|_{\bL^2(\sO)}) + \gamma_M^{\frac{1}2}\|\bv^{k}\|_{\bL^2(0,L)}+  C(\delta)\|\bv^k\|_{\bL^2(0,L)}.
\end{align*} 

Therefore, we obtain for some $C>0$ that depends only on $T$ and $\delta$ that,
\begin{align}	
\notag	\bE[(\Delta t) \sup_{0\leq n\leq N}&\sum_{k=n-j+1}^n \|\bu^{k,n}_{\sigma,M}\|^2_{\bH^1(\sO)}]\\ \notag &\leq C(\delta, T)\bE\left( (\Delta t) \sum_{n=0}^N\|\bu^k\|^2_{\bH^1(\sO)} + ({\sigma^{-2}}+ \gamma_M^{{}})\sup_{0\leq n\leq N}\|\bv^{k}\|^2_{\bL^2(0,L)}\right)\\
\label{uh1}	& \leq C(\delta,T)(1+{\sigma^{-2}}+ \gamma_M^{}).
	\end{align}
Similarly (see also \eqref{lambda}),
\begin{align}\label{vh2}
	\|\bv^{k,n}_{\sigma,M}\|_{\bH^2(0,L)}
	\leq  \gamma_M^{}\|\bv^{k}\|_{\bL^2(0,L)} + C(\delta)\|\bv^k\|_{\bL^2(0,L)}.
\end{align}
Then, as in \cite{G08}, for any $n \leq N$ we "test" \eqref{first}$_3$ and \eqref{second}$_2$ with \begin{align}	\label{Qn}
	\bQ_n:=(\bq_n,\bfpsi_n)=\left( (\Delta t) \sum_{k=n-j+1}^n\bu^{k,n}_{\sigma,M},\,\, (\Delta t) \sum_{k=n-j+1}^n\bv_{\sigma,M}^{k,n}\right) .
\end{align}This gives us
\begin{small}
	\begin{align*}	 	
		&\int_{\sO}\left((J^{n+1}_*) \bu^{n+1}-(J^{n}_*)\bu^{n}\right) \left( \Delta t \sum_{k=n-j+1}^n\bu^{k,n}_{\sigma,M}\right)  
		+\int_0^L(\bv^{n+1}-\bv^{n} )\left( \Delta t \sum_{k=n-j+1}^n\bv_{\sigma,M}^{k,n}\right) \\
		&\frac{-1}2\int_\sO \left( J_*^{n+1}-J_*^n\right)  \bu^{n+1}\cdot \left( \Delta t \sum_{k=n-j+1}^n\bu^{k,n}_{\sigma,M}\right) +(\Delta t) b^{\bfeta_n^*}(\bu^{n+1},\bw^n_*,\left( \Delta t \sum_{k=n-j+1}^n\bu^{k,n}_{\sigma,M}\right) ) \\
		&+ \frac{(\Delta t)}{\alpha}\int_\Gamma S^n_* (\bu^{n+1}-\bv^{n+1})\left( (\Delta t) \sum_{k=n-j+1}^n\bu^{k,n}_{\sigma,M}-\bv_{\sigma,M}^{k,n}\right) \\
		& + \frac{(\Delta t)}{\ep}\int_\sO 
		\text{div}^{\bfeta^n_*}\bu^{n+1}\text{div}^{\bfeta^n_*}\left( \Delta t \sum_{k=n-j+1}^n\bu^{k,n}_{\sigma,M}\right) d\bx  
		\end{align*}\begin{align*}
		&+ \frac{(\Delta t)}{\ep}\int_\Gamma (\bu^{n+1}-\bv^{n+1})\cdot\bn^n_*\left( (\Delta t) \sum_{k=n-j+1}^n\bu^{k,n}_{\sigma,M}-\bv_{\sigma,M}^{k,n}\right)\cdot\bn^n_*   \\
		&	+2\nu(\Delta t)\int_{\sO}(J^n_*) \bD^{\eta_*^{n}}(\bu^{n+1})\cdot \bD^{\eta_*^{n}}\left( \Delta t \sum_{k=n-j+1}^n\bu^{k,n}_{\sigma,M}\right) \\
			&+(\Delta t)\left(  \sL_e(\bfeta^{n+\frac12}) , \left( \Delta t \sum_{k=n-j+1}^n\bv_{\sigma,M}^{k,n}\right)  +\ep\int_0^L\partial_{zz}\bv^{n+\frac12}\cdot\partial_{zz}\left( \Delta t \sum_{k=n-j+1}^n\bv^{k,n}_{\sigma,M}\right)\right) \\
		&= (\Delta t)\left( P^n_{{in}}\int_{0}^1q^n_z\Big|_{z=0}dr-P^n_{{out}}\int_{0}^1q^n_z\Big|_{z=1}dr\right)
		+(G(\bu^{n},\bfeta_*^n)\Delta_n W, \bQ_n).
	\end{align*}
	\end{small}
	Observe that summation by parts formula gives us for the first two terms,
	\begin{align*}
		&=(\Delta t)\sum_{n=1}^N\left( \int_{\sO}(J^{n}_*) \bu^{n}(\bu^{n,n}_{\sigma,M}-\bu_{\sigma,M}^{n-j,n})d\bx+\int_0^L \bv^n(\bv_{\sigma,M}^{n,n}-\bv_{\sigma,M}^{n-j,n})dz\right)\\
		&- \int_\sO (J^{N+1}_*)\bu^{N+1} \left( \Delta t \sum_{k=N-j+1}^N\bu^{k,N}_{\sigma,M}\right)-\int_0^L \bv^N \left( \Delta t \sum_{k=N-j+1}^N\bv_{\sigma,M}^{k,n}\right),
	\end{align*}
		{where the first term on the right side above can be written as,}
		\begin{align*}
		&(\Delta t)\sum_{n=0}^N\left(
		\int_{\sO}(J^{n}_*) \bu^{n}(\bu^{n}-\bu^{n-j})d\bx+ \int_{\sO}(J^{n}_*) \bu^{n}(\bu^{n,n}_{\sigma,M}-\bu^{n}-(\bu_{\sigma,M}^{n-j,n}-\bu^{n-j}))d\bx\right)\\
	&+	\int_0^L \bv^n\left( \bv_{M}^{n}-\left( \frac{\lambda^{n}_\sigma-\lambda^{n,n}_M}{\lambda_0^n}\right)\xi_0 -\bv_{M}^{n-j}+\left( \frac{\lambda^{n-j}_\sigma-\lambda^{n-j,n}_M}{\lambda_0^n}\right)\xi_0\right) dz\\
		&=(\Delta t)\sum_{n=0}^N\left(
		\frac12\int_{\sO}(J^{n}_*) (|\bu^{n}|^2-|\bu^{n-j}|^2+|\bu^{n}-\bu^{n-j}|^2)d\bx\right) \\
		&+(\Delta t)\sum_{n=0}^N \left( \int_{\sO}(J^{n}_*) \bu^{n}(\bu^{n,n}_{\sigma,M}-\bu^{n}-(\bu_{\sigma,M}^{n-j,n}-\bu^{n-j}))d\bx\right)\\
		&+\frac{(\Delta t)}2\sum_{n=0}^N\int_0^L |\bv_M^n|^2-|\bv_M^{n-j}|^2+|\bv_M^n-\bv_M^{n-j}|^2- \bv^n(\frac{\lambda^{n}_\sigma-\lambda^{n-j}_\sigma+\lambda^{n-j,n}_M-\lambda^{n,n}_M}{\lambda_0^n})\xi_0
	\end{align*}	 
	where we set $\bu^n=0$ and $\bv^n=0$ for $n<0$ and $n>N$.
	Observe that since
	\begin{align*}
		\|\bv^n-\bv^n_M-\bv^{n-j}-\bv^{n-j}_M\|_{\bH^{-\beta}(0,L)} \leq \gamma_M^{-\frac{\beta}2}(\|\bv^n\|_{\bL^2(0,L)}+\|\bv^{n-j}\|_{\bL^2(0,L)}),
	\end{align*}
	for any $0\leq \beta<\frac12$, the right hand side of the preceding equation will give us the desired terms \eqref{shift}. In what follows we will treat rest of the terms. First,
	\begin{align*}
		I_0:=	\Delta t\sum_{n=0}^N\int_0^L |\bv^n_M|^2-|\bv^{n-j}_M|^2 =\Delta t \sum_{n=N-j+1}^{N}\|\bv^n_M\|^2_{\bL^2(0,L)}\leq h\max_{0\leq n \leq N}\|\bv^n\|^2_{\bL^2(0,L)}.
	\end{align*} 
	Hence,
$
		\bP(|I_0|>R) \leq \frac{h}R \bE[\max_{0\leq n \leq N}\|\bv^n\|^2_{\bL^2(0,L)}] \leq \frac{Ch}{R}.$
		
Next, we recall that for two matrices $A$ and $B$, the derivative of the determinant $D(det)(A)B=det(A)tr(BA^{-1})$. Hence applying the mean value theorem to $det(\nabla A_{\bfeta^n_*})-det(\nabla A_{\bfeta^{n+j}_*})$, using \eqref{boundJ} and that $det(A) \leq (\max A_{ij})^2$, we obtain, for some $\beta\in[0,1]$, that (for more details see (73) in \cite{MC16}):
	\begin{equation}
		\begin{split}\label{boundJt}
			\|{J^{n+j}_*-J^{n}_*}\|_{L^\infty(\sO)}&= \|\text{det}(\nabla A^{n,\beta})\nabla^{n,\beta}\cdot(A_{\bfeta^{n+j}_*}-A_{\bfeta^{n}_*})\|_{L^\infty(\sO)}\\
			& \leq C(\delta)\|A_{\bfeta^{n+j}_*}-A_{\bfeta^{n}_*}\|_{\bC^1(\sO)},
		\end{split}
	\end{equation}
	where $\nabla^{n,\beta}=\nabla^{\bfeta^n_*} +\beta(\nabla^{\bfeta^{n+j}_*}-\nabla^{\bfeta^n_*})$ and $\nabla A^{n,\beta}=\nabla A_{\bfeta^n_*} +\beta(\nabla A_{\bfeta^{n+j}_*}-\nabla A_{\bfeta^n_*})$.\\
Using \eqref{boundJ} again, we find the following bounds for any $s>\frac32$,
	\begin{align*}
		I_1&:=	(\Delta t)\sum_{n=0}^N\left(
		\int_{\sO}(J^{n}_*) (|\bu^{n}|^2-|\bu^{n-j}|^2)d\bx\right)\\
		&=	(\Delta t)\left( \sum_{n=N-j+1}^N\int_{\sO}(J^{n}_*)|\bu^n|^2 d\bx	+\sum_{n=0}^{N-j}\int_{\sO}(J^{n}_*-J^{n+j}_*) |\bu^{n}|^2d\bx\right) \\
		&\leq	(\Delta t) \sum_{n=N-j+1}^N\int_{\sO}(J^{n}_*)|\bu^n|^2 d\bx	+	(\Delta t)\sum_{n=0}^{N-j}\| A_{\bfeta^{n+j}_*}- A_{\bfeta^n_*}\|_{\bC^1(\sO)} \|\bu^{n}\|^2_{\bL^2(\sO)} \\
		&\leq   \left( h\sup_{1\leq n \leq N}\int_{\sO}(J^{n}_*)|\bu^n|^2 d\bx+ \sup_{0\leq k \leq N-1}\|\bu^{n}\|^2_{\bL^2(\sO)}h^{\frac14}\| \bfeta^{n+j}_*-\bfeta^n_*\|_{\bH^s(0,L)} \right)\\
		&\leq  h^{\frac14}  \sup_{0\leq k \leq N-1}\|\bu^{n}\|^2_{\bL^2(\sO)}\left(1+\|\tilde \bfeta^*_N\|_{C^{0,\frac14}(0,T;\bH^s(0,L))} \right).		
	\end{align*}
	In what follows, we will repeatedly make use of the fact that for any two positive random variables $X$ and $Y$, $\{\omega: X+Y>R\} \subseteq \{\omega: X>\frac{R}2\}\cup \{\omega: Y>\frac{R}2\}$ which further implies that
	\begin{align*}
		\bP(|X|+|Y|>R) \leq \bP(|X|>\frac{R}2)+\bP(|Y|>\frac{R}2),
		\shortintertext{and similarly,}
		\bP(|XY|>R) \leq \bP(|X|>\sqrt{R})+\bP(|Y|>\sqrt{R}).
	\end{align*}
The embedding $W^{1,\infty}(0,T;L^2(0,L))\cap L^2(0,T;H^2_0(0,L)) \hookrightarrow C^{0,\frac14}(0,T;H^\frac32(0,L))$, then gives us for any $s\geq \frac32$ that
	\begin{align*}
		\bP&\left(|I_1|>R \right) \leq \bP\left( h^{\frac18}\sup_{1\leq n \leq N}\|\bu^n\|^2_{\bL^2(\sO)} \geq \sqrt{R}\right) + \bP\left( h^{\frac18}\|\tilde\bfeta^*_N\|_{C^{0,\frac14}(0,T;\bH^s(0,L))}\geq \sqrt{R} \right)\\
		&\leq \frac{{h}^{\frac18}}{\sqrt{R}}\bE\left( \sup_{1\leq n \leq N}\|\bu^n\|^2_{\bL^2(\sO)}  +\|\tilde\bfeta^*_N\|_{C^{0,\frac14}(0,T;\bH^s(0,L))} \right)\\
		&\leq \frac{{h}^{\frac18}}{\sqrt{R}}\bE\left( \sup_{1\leq n \leq N}\|\bu^n\|^2_{\bL^2(\sO)}  +\|\tilde\bfeta^*_N\|_{W^{1,\infty}(0,T;\bL^2(0,L))\cap L^2(0,T;\bH^2_0(0,L)) } \right) \leq C\frac{{h}^{\frac18}}{\sqrt{R}}.
	\end{align*}
	Next, thanks to \eqref{u1}, we see that
	\begin{align*}
		& I_2=	|(\Delta t)\sum_{n=0}^N \left( \int_{\sO}(J^{n}_*) \bu^{n}(\bu^{n,n}_{\sigma,M}-\bu^{n}-(\bu_{\sigma,M}^{n-j,n}-\bu^{n-j}))d\bx\right)|\\
		& \leq C(\delta)(\Delta t)\sum_{n=0}^N\|\bu^n\|_{\bL^2(\sO)}(\|\bu^{n,n}_{\sigma,M}-\bu^n\|_{\bL^2(\sO)}+\|\bu^{n-j,n}_{\sigma,M}-\bu^{n-j}\|_{\bL^2(\sO)})\\
	&\lesssim  \sup_{1\leq n\leq N} \|\bu^n\|_{\bL^2(\sO)}\sum_{k=0}^N \Big(     h^{\frac14}\|\tilde\bfeta^*_N\|_{C^{0,\frac14}(0,T;\bH^s(0,L))}(\|\bu^k\|_{\bL^2(\sO)}+\|\bv^k\|_{\bL^2(0,L)} )\\
	&+\sigma\|\bu^k\|_{\bH^1(\sO)}
	 +	(\sigma^{\beta}+\gamma_M^{-\frac{\beta}2}+h^{\frac14}) \|\bv^k\|_{\bL^2(0,L)}\Big).
	\end{align*}
Similarly, thanks to \eqref{lambda}, we see for any $0\leq \beta<\frac12$ that
\begin{align*}
I_3&:=|(\Delta t)\sum_{n=0}^N	\int_0^L \bv^n(\frac{\lambda^{n,n}_\sigma-\lambda^{n-j,n}_\sigma+\lambda^{n-j,n}_M-\lambda^{n,n}_M}{\lambda_0^n})\xi_0|\\
&\leq C(\delta)\left( \|\bfeta^n_*\|_{\bH_0^2(0,L)}\sigma^{\frac12}+\sigma^{\beta}+{\gamma_M^{-\frac{\beta}2}}+h^{\frac14}\|\tilde\bfeta^*_N\|_{C^{0,\frac14}(0,T;\bH^s(0,L))}\right) \|\bv^n\|_{\bL^2(0,L)}.
\end{align*}
	Hence, we obtain that,
$$\bP(|I_2+I_3|>R)	\leq \frac{C}R(h^{\frac14}+\sigma^{\beta}+\gamma_M^{-\frac{\beta}2}) \leq  \frac{C}{R}h^{\frac{\beta}{12}}.$$
	For the penalty term, thanks to \eqref{kton}$_1$, we have,
\begin{align*}
	& \bE[I_4]:=\bE|	\frac{(\Delta t)}{\ep}\sum_{n=0}^{N}\int_\sO 
	\text{div}^{\bfeta^n_*}\bu^{n+1}\left( \Delta t \sum_{k=n-j+1}^n\text{div}^{\bfeta^n_*}(\bu^{k,n}_{\sigma,\lambda})\right) d\bx |  \\
	&\leq \bE\frac{(\Delta t)}{{\ep}}\sum_{n=0}^{N}
	\|	\text{div}^{\eta^n_*}\bu^{n+1}\|_{L^2(\sO)}\left( 	{(\Delta t)}{{}} \sum_{k=n-j+1}^n\|\text{div}^{\bfeta^n_*}(\bu^{k,n}_{\sigma,\lambda})\|^2_{L^2(\sO)}\right)^{\frac12} \sqrt{h} \end{align*}\begin{align*}
	&\leq C \sqrt{hT} \bE \left(  	\frac{(\Delta t)}{\ep} \sum_{n=0}^N\|\text{div}^{\bfeta^{n}_*}(\bu^{n+1})\|^2_{L^2(\sO)}\right) \leq Ch^{\frac12} .
\end{align*}
	Notice that, due to Lemma \ref{bounds} (5), the constant $C$ in the estimate above does not depend on $\ep$. The other penalty term is handled identically thanks to \eqref{kton}$_2$ and Lemma \ref{bounds} (6). 
	Similarly, due to \eqref{uh1}, we have
	\begin{align*}
		\bE[I_5] =&\bE|	{(\Delta t)}{}\sum_{n=0}^{N}\int_\sO 
		\bD^{\eta^n_*}\bu^{n+1}\left( \Delta t \sum_{k=n-j+1}^n\bD^{\eta^n_*}(\bu^{k,n}_{\sigma,M})\right) d\bx |  \\
		&\leq \sqrt{h}\bE((\Delta t)\sum_{n=0}^{N}
		\|\bu^{n+1}\|_{\bH^1(\sO)}\left( {(\Delta t)} \sum_{k=n-j+1}^n\|(\bu^{k,n}_{\sigma,M})\|^2_{\bH^1(\sO)}\right)^{\frac12}) \\
		&\leq \sqrt{h}\bE\sum_{n=0}^{N}(\Delta t)
		\|\bu^{n+1}\|^2_{\bH^1(\sO)}+\sqrt{h}\bE\left( {(\Delta t)}\sup_{0\leq n\leq N} \sum_{k=0}^N\|(\bu^{k,n}_{\sigma,M})\|^2_{\bH^1(\sO)}\right)^{} \\		
	& \leq  C(\delta)h^{\frac14}(1+{\sigma^{-2}}+ \gamma_M^{}) \leq Ch^{\frac1{12}}.
	\end{align*}	
Next, using the embedding $H^{\frac12}(\sO) \hookrightarrow L^4(\sO)$ we obtain, for some $C>0$ which depends only on $\delta$, that
\begin{align*}
	&I_6:=|	(\Delta t)\sum_{n=0}^N b^{\bfeta^n_*}\left( \bu^{n+1},\bw_*^n,\left( \Delta t \sum_{k=n-j+1}^n\bu^{k,n}_{\sigma,M}\right) \right) |\\
	&\leq\sqrt{h}	(\Delta t)\sum_{n=0}^N  \|\bu^{n+1}-\bw_*^{n}\|_{\bL^4(\sO)} \|\bu^{n+1}\|_{\bL^4(\sO)} \left( \Delta t \sum_{k=n-j+1}^n\|\bu^{k,n}_{\sigma,M}\|^2_{\bH^1(\sO)}\right)^{\frac12}\\
	& \leq\sqrt{h}	(\Delta t)\sum_{n=0}^N  \left(\|\bu^{n+1}\|_{\bH^1(\sO)}+\|\bw_*^{n}\|_{\bH^\frac12(0,L)}\right) \|\bu^{n+1}\|_{\bH^1(\sO)}\\
	&\qquad\qquad  \left( \Delta t \sum_{k=n-j+1}^n
	\|\bu^{k,n}_{\sigma,M}\|^2_{\bH^1(\sO)}\right)^{\frac12}.
\end{align*}
Again, thanks to \eqref{uh1}, for any $R>0$, we have
\begin{align*}
		\bP\left( |	I_6 | > R\right) &\leq \frac{h^{\frac14}}{\sqrt{R}}\bE[(\Delta t)\sum_{n=0}^N  \|\bu^{n+1}\|^2_{\bH^1(\sO)}] +\frac{h^{\frac14}}{\sqrt{R}}\bE[\sup_{0\leq n\leq N-1}\|\bv^{n+\frac12}\|^2_{\bL^2(0,L)}] \\
		&+  \frac{h^{\frac12}}{R} \bE[(\Delta t)^2\sum_{n=0}^N\sum_{k=n-j+1}^n
			\|\bu^{k,n}_{\sigma,M}\|^2_{\bH^1(\sO)}] \\
			&\leq 		\frac{C(\delta)h^{\frac14}}{\sqrt{R}}+\frac{C(\delta)h^{\frac12}}{{R}}(1+{\sigma^{-2}}+ \gamma_M^{})\leq\frac{C(\delta)}{\sqrt{R}}h^{\frac1{4}} .
\end{align*}
Next, calculations similar to \eqref{boundJt}, give us
\begin{align}\label{boundJt2}
	\|\frac{J^{n+1}_*-J^{n}_*}{\Delta t}\|_{L^2(\sO)}\leq C(\delta)\|\bw^{n+1}_*\|_{\bH^1(\sO)}\leq C(\delta)\|\bv_*^{n+1}\|_{\bH^{\frac12}(0,L)}.
\end{align}
Hence,
\begin{align*}
&	I_7:=	|\sum_{n=0}^N\int_\sO \left( J_*^{n+1}-J_*^n\right)  \bu^{n+1}\cdot \left( \Delta t \sum_{k=n-j+1}^n\bu^{k,n}_{\sigma,M}\right) d\bx | \\
	&\leq (\Delta t)\sum_{n=0}^N\|\bv^{n+\frac12}\|_{\bH^\frac12(0,L)}\|\bu^{n+1}\|_{\bL^6(\sO)}\left( (\Delta t)\sum_{k=n-j+1}^n    \|  \bu^{k,n}_{\sigma,M} \|^2_{\bL^3(\sO)}\right) ^{\frac12}\sqrt{h}\\
	&\leq C\sqrt{h} \left((\Delta t) \sum_{n=0}^N\|\bv^{n+\frac12}\|^2_{\bH^\frac12(0,L)} \right) ^{\frac12} \left( (\Delta t)\sum_{n=0}^N    \|  \bu^{n} \|^2_{\bH^1(\sO)}\right)^{\frac12}\\
	&\qquad\left( (\Delta t)\sup_{0\leq n\leq N}\sum_{k=n-j+1}^n    \|  \bu^{k,n}_{\sigma,M} \|^2_{\bH^1(\sO)}\right) ^{\frac12}.
\end{align*}
This implies, thanks to \eqref{uh1}, that,
	\begin{align}
	\bP\left( |	I_7 | \geq R^{}\right) 
\label{ep2}	&\leq  \frac{C}{\ep}\frac{h^{\frac13}}{R^{\frac23}} +\frac{C(\delta)h^{\frac13}}{R^{\frac23}}(1+{\sigma^{-2}}+ \gamma_M) \leq{C(\ep)}\frac{h^{\frac1{6}}}{R^{\frac23}}.	
\end{align}
Note that this estimate depends on $\ep$.	Next, using \eqref{vh2}, 
	we obtain that,
	\begin{align*}
		&\bE[I_8]:=	\bE|(\Delta t)\sum_{n=0}^N	\left(   \sL_e(\bfeta^{n+1} ), \left( \Delta t \sum_{k=n-j+1}^n\bv^{k}_{\sigma,M}\right)\right) |\\
		&\qquad\leq C(\delta) (\Delta t)^2 \bE\sum_{n=0}^N\left( \|\partial_{zz}\bfeta^n \|_{\bL^2(0,L)} \sum_{k=n-j-1}^n\|\partial_{zz}\bv^{k,n}_{\sigma,M}\|_{\bL^2(0,L)} \right)  \\
		&\qquad\leq \bE(\Delta t)^2 \sum_{n=0}^N\left( \|\partial_{zz}\bfeta^n \|_{\bL^2(0,L)} \sum_{k=n-j-1}^n \gamma_M^{}\|\bv^{k}\|_{\bL^2(0,L)}\right)  \\
		&\qquad\leq {CT}{h^{}}\gamma_M\bE\max_{0\leq n\leq N}\left( \|\bfeta^n \|^2_{\bH^2(0,L)}+ \|\bv^n\|^2_{\bL^2(0,L)} \right) \leq Ch^{\frac56}.
	\end{align*}
Finally, we treat the stochastic term (see also \eqref{tower}) using Young's inequality:
	\begin{align*}
		&\bE\left( \sum_{n=0}^N	|(G(\bu^{n},\bfeta_*^n)\Delta_n W, \bQ_n)|\right)  
			\leq 		\bE\Big( \sum_{n=0}^N	\|G(\bu^{n},\bfeta_*^n)\|_{L_2(U_0;\bL^2)}\|\Delta_n W\|_{U_0} \\
			&\qquad\times \left( (\Delta t) \sum_{k=n-j+1}^n\|\bu^{k,n}_{\sigma,M}\|_{\bL^2(\sO)}^2 +\|\bv_{\sigma,M}^k\|_{\bL^2(0,L)}^2\right)^{\frac12}h^{\frac12}\Big) 	\\
		& \qquad	\leq 	h^{\frac12}	\bE \sum_{n=0}^N	\|G(\bu^{n},\bfeta_*^n)\|^2_{L_2(U_0;\bL^2)}\|\Delta_n W\|^2_{U_0}  \\
		&\qquad+\bE\left( (\Delta t) \sum_{k=n-j+1}^n\|\bu^{k,n}_{\sigma,M}\|_{\bL^2(\sO)}^2 +\|\bv_{\sigma,M}^k\|_{\bL^2(0,L)}^2\right) 
		 \leq Ch^{\frac12}.
	\end{align*}
	Now to show that the laws of the random variables mentioned in the statement of the theorem are tight, we will consider the following sets for $0\leq \alpha<1$ and $0<\beta<\frac12$ and any $R>0$,
	\begin{align*}
		&{\mathcal{B}}_R:=\{(\bu,\bv)\in L^2(0,T;\bH^1(\sO))\times L^2(0,T;\bL^{2}(0,L)):\|{\bu}\|^2_{L^2(0,T;\bH^1(\sO))}+\\
		&\|{\bv}\|^2_{L^2(0,T;\bL^{2}(0,L))}+\sup_{0<h<1}{h^{-\frac{\beta}{12}}}\int_h^{T}\left( \|T_h\bu-\bu\|^2_{\bL^2(\sO)}+\|T_h\bv-\bv\|^2_{\bH^{-\beta}(0,L)}\right)  \le R\}.
	\end{align*}
Thanks to Lemma \ref{compactLp}, $\sB_R$ is compact in $L^2(0,T;\bH^{\alpha}(\sO))\times L^2(0,T;\bH^{-\beta}(0,L))$ for each $R>0$, $0\le\alpha<1$ and $0<\beta<\frac12$.
	Now we apply Chebyshev's inequality to obtain the desired result:
	\begin{align*}
		\bP((\bu^+_N,&\bv^+_N) \notin \sB_R)\leq \bP\left( \|{\bu^+_N}\|^2_{L^2(0,T;\bH^1(\sO))}+\|{\bv^+_N}\|^2_{L^2(0,T;\bL^{2}(0,L))}>\frac{R}2\right) \\
		&+\bP\left( \sup_{0 <h<1 }{h^{-\frac{\beta}{12}}}\int_h^{T}\left( \|T_h\bu^+_N-\bu^+_N\|^2_{\bL^2(\sO)}+\|T_h\bv^+_N-\bv^+_N\|^2_{\bH^{-\beta}(0,L)}\right)  > \frac{R}2\right) \\
		&\leq \frac{C}{\sqrt{R}},
	\end{align*}
	where $C>0$ depends only on $\delta$, Tr$Q$, $\ep$ and the given data and is independent of $N$.
\end{proof}
Next we give the tightness results for the structure displacements and velocities.
\begin{lem}\label{tightl2}
	For fixed $\delta$, the following statements hold:
	\begin{enumerate}
		\item The laws of $\{\tilde\bfeta_{N}\}_{N\in \mathbb{N}}$ and that of $\{\tilde\bfeta^*_{N}\}_{N\in \mathbb{N}}$ are tight in $C([0,T],\bH^{s}(0,L))$ for any $s<2$.
			\item The laws of $\{\|\bu^+_N\|_{L^2(0,T;V)}\}_{N\in \mathbb{N}}$ are tight in $\R$.
		\item For a fixed $\ep$, the laws of $\{\|\bv^*_N\|_{L^2(0,T;\bH_0^1(0,L))}\}_{N\in \mathbb{N}}$ are tight in $\R$.
	\end{enumerate}
\end{lem}
\begin{proof}
The Aubin-Lions theorem gives us: For $0<s<2$,
	$$L^\infty(0,T;\bH_0^2(0,L))\cap W^{1,\infty}(0,T;\bL^2(0,L)) \subset\subset C([0,T];\bH^{s}(0,L)). 
	$$
	Hence for any $R>0$ we consider
	\begin{align*}\mathcal{K}_R:=&\{{\bfeta}\in L^\infty(0,T;\bH^2_0(0,L))\cap W^{1,\infty}(0,T;\bL^2(0,L)):\\
		&\|{\bfeta}\|^2_{L^\infty(0,T;\bH^2_0(0,L))} + \|{\bfeta}\|^2_{ W^{1,\infty}(0,T;\bL^2(0,L))}\leq R\}.\end{align*}
	Using the Chebyshev inequality and Lemma \ref{bounds} we obtain for some $C>0$ independent of $N$ and $\ep$ that the following holds:
	\begin{equation*}
		\begin{aligned}
			\bP\left[\tilde{\bfeta}_{N}\not\in {\mathcal{K}}_R\right]&\le \bP\left[
			{\|\tilde{\bfeta}_{N}\|}^2_{L^\infty(0,T;\bH^2_0(0,L))}\ge \frac{R}{2}\right]+
			\bP\left[
			{\|\tilde{\bfeta}_{N}\|}^2_{W^{1,\infty}(0,T;\bL^2(0,L))}\ge \frac{R}{2}\right]	\\
			&\le  \frac{4}{R^2}\bE\left[
			{\|\tilde{\bfeta}_{N}\|}^2_{L^\infty(0,T;\bH^2_0(0,L))}+\|\tilde{\bfeta}_{N}\|^2_
			{W^{1,\infty}(0,T;\bL^2(0,L))}\right]
			\leq \frac{C}{R^2}.
		\end{aligned}
	\end{equation*}
	The proof of the statements (2), (3) follow by an identical application of the Chebyshev inequality and the bounds obtained in Lemma \ref{bounds}. For any $R>0$,
	\begin{align*}
		\bP[\| \bv^*_N\|_{L^2(0,T;\bH_0^1(0,L))}>R] \leq \frac1{R^2}\bE[\| \bv^*_N\|^2_{L^2(0,T;\bH_0^1(0,L))}] \leq \frac{C(\ep)}{R^2}.
	\end{align*}
	This completes the proof of Lemma \ref{tightl2}.
\end{proof}
	To obtain almost sure convergence of the rest of the random variables we will use the following lemma which is a consequence of the bounds on numerical dissipation.
	\begin{lem}\label{difference}
	The following convergence results hold: 
		\begin{enumerate}
			\item $\lim_{N\rightarrow\infty} \bE\int_0^T\|\bu_{N}-\bu^+_{N}\|^2_{\bL^2(\sO)}dt=0$,
			\, $\lim_{N\rightarrow\infty} \bE\int_0^T\|\bu_{N}-\tilde\bu_{N}\|^2_{\bL^2(\sO)}dt=0$.
			\item  $\lim_{N\rightarrow\infty} \bE\int_0^T\|\bv_{N}-\tilde \bv_{N}\|^2_{\bL^2(0,L)}dt=0$,
			 $\lim_{N\rightarrow\infty} \bE\int_0^T\|\bv_{N}-\bv^{\#}_{N}\|^2_{\bL^2(0,L)}dt=0$.
			\item $\lim_{N\rightarrow\infty} \bE\int_0^T\|{\bfeta}_{N}-\tilde{\bfeta}_{N}\|^2_{\bH_0^2(0,L)}dt=0$,
			 $\lim_{N\rightarrow\infty} \bE\int_0^T\|{\bfeta}^+_{N}-\tilde{\bfeta}_{N}\|^2_{\bH_0^2(0,L)}dt=0$.
		\end{enumerate}
	\end{lem}
	\begin{proof} Statement (1)$_1$ is true thanks to Theorem \ref{energythm} (3). We prove (1)$_2$ below exactly as in \cite{TC23}. 
		\begin{align*}
		&	\bE\int_0^T\|\bu_{N}-\tilde\bu_{N}\|^2_{\bL^2(\sO)}dt\\
		&=
			\bE\sum_{n=0}^{N-1}\int_{t^n}^{t^{n+1}}\frac1{\Delta t}\|(t-t^n)\bu^{n+1}+(t^{n+1}-t-\Delta t)\bu^{n}\|^2_{\bL^2(\sO)}dt\\
			&=\bE\sum_{n=0}^{N-1}\|\bu^{n+1}-\bu^n\|_{\bL^2(\sO)}^2\int_{t^n}^{t^{n+1}} \left( \frac{t-t^n}{\Delta t}\right) ^2dt
			\leq \frac{CT}{\delta_1 N} \rightarrow 0 \quad N\rightarrow \infty.
		\end{align*}
		The rest follows identically from the uniform estimates stated in Theorem \ref{energythm}.
	\end{proof}
	To pass $N\to\infty$  in the weak formulation, we consider the following random variable:
	\begin{align}\label{UN}
		\tilde\bU_N(t):=(	(\tilde J_N^*(t))\tilde\bu_N(t),\tilde \bv_N(t))-E_N(t),
	\end{align}
	where $E_N$ is an error term that appears due to discretizing the stochastic integral (see \eqref{approxsystem}) given by:
	$$E_N(t)=\sum_{m=0}^{N-1}\left( \frac{t-t^m}{\Delta t}G(\bu^m,\bfeta^m_*)\Delta_mW-\int_{t^m}^tG(\bu^m,\bfeta^m_*)dW\right) \chi_{[t^m,t^{m+1})}.$$
	Let $	\sU_1:=\sU\cap (\bH^2(\sO)\times \bH^3_0(0,L)).$
	For any $\sV_1\subset\subset \sU_1$, we denote by $\mu^{u,v}_N$ the probability measure of $\tilde\bU_N$: $$\mu^{u,v}_{N}:=\bP\left( \tilde\bU_{N}\in \cdot \right) \in Pr(C([0,T];\sV'_1)) ,$$
	where $Pr(S)$, denotes the set of probability measures on a metric space $S$. Then we have the following tightness result which is proven identically as Lemma 4.6 in \cite{TC23}.
	\begin{lem}\label{uvtightC}
		For a fixed $\ep>0$ and $\delta$, the laws $\{\mu_N^{u,v}\}_N$ of {the random variables } $\{\tilde\bU_N\}_{N}$ 
		are tight in 
		$C([0,T];\sV_1')$.
	\end{lem}
	We note here that this result is available only in the case of fixed $\ep>0$ and that we will not have this result in the next section.
In the following lemma, we will show that the stochastic error term vanishes as $N\to \infty$.
	\begin{lem}\label{u*diff}
		{The numerical error $E_N$ of the stochastic term has the following property:}
		$$ \bE\int_0^T\|E_N(t)\|^2_{\bL^2(\sO)\times \bL^2(0,L)}dt  \rightarrow 0 \text{ as } {N\rightarrow\infty}.$$
	\end{lem}
	\begin{proof}
	For any $N$, we begin by writing
		$E_N(t)
		=:E_N^1+E_N^2$. 	Observe that $E^1_N$ satisfies,
		\begin{align*}
			&\bE\int_0^T\|E_N^1(t)\|_{\bL^2}^2dt	=\bE\sum_{n=0}^{N-1} \|G(\bu^n,\bfeta^n_*)\Delta_nW\|_{\bL^2}^2 \int_{t^n}^{t^{n+1}}|\frac{t-t^m}{\Delta t}|^2dt\\
			&=\bE\sum_{n=0}^{N-1} \|G(\bu^n,\bfeta^n_*)\Delta_nW\|_{\bL^2}^2 \frac{\Delta t}{3 }\leq \bE\sum_{n=0}^{N-1} \|G(\bu^n,\bfeta^n_*)\|_{L_2(U_0;\bL^2)}^2\|\Delta_nW\|_{U_0}^2 {\Delta t}\\
			&	= (\Delta t)^2\bE\sum_{n=0}^{N-1} \|G(\bu^n,\bfeta^n_*)\|_{L_2(U_0;\bL^2)}^2 \leq C \Delta t,
		\end{align*}
		where, as a consequence of Theorem \ref{energythm}, $C>0$ is independent of $N$ and $\ep$.
		
		To estimate $E^2_N$ we use the It\^{o} isometry as follows
		\begin{align*}
			&\bE\int_0^T\|E_N^2(t)\|_{\bL^2}^2dt= \bE\sum_{n=0}^{N-1}\int_{t^n}^{t^{n+1}}\|\int_{t^n}^tG(\bu^n,\bfeta_*^n)dW\|_{\bL^2}^2dt\\
			&= \bE\sum_{n=0}^{N-1}\int_{t^n}^{t^{n+1}}\int_{t^n}^t\|G(\bu^n,\bfeta_*^n)\|_{L_2(U_0;\bL^2)}^2dsdt
			\\&= \frac12\bE\sum_{n=0}^{N-1}\|G(\bu^n,\bfeta_*^n)\|_{L_2(U_0;\bL^2)}^2(\Delta t)^2 \leq C\Delta t.
		\end{align*}	
	\end{proof} 

	\subsection{Almost sure convergence}\label{almostsure}
	In this section we will obtain almost sure convergence results for our approximate solutions.
To that end, let 
$\mu_{N}$ be the joint law of the approximate  random variables
$\mathcal{U}_{N}=(\bu^+_{N},\bv^+_{N},\tilde\bfeta^*_{N},\tilde\bfeta_{N}, \|\bu^+_N\|_{L^2(0,T;V)}, \|\bv^*_N\|_{L^2(0,T;\bH_0^1(0,L))},\tilde\bU,W)$ taking values in the phase space
\begin{align*}
	\Upsilon&:=[L^2(0,T;\bL^{2}(\sO))]\times[L^2(0,T;\bH^{-\beta}(0,L))]\times [C([0,T],\bH^s(0,L))]^2\times\R^2\\
	& \times 
	 C([0,T];\sV_1')\times C([0,T];U),
\end{align*}
for some fixed $\frac32<s<2$ and $0<\beta<\frac12$.

Since $C([0,T];U)$ is separable and metrizable by a complete metric,  the sequence of Borel probability measures, $\mu^W_{N}(\cdot):=\bP(W\in\cdot)$, that are constantly equal to one element, is tight on $C([0,T];U)$. Thus, recalling Lemmas \ref{tightuv} \ref{tightl2}, \ref{u*diff} and using Tychonoff's theorem it follows that the sequence of the probability measures $\mu_{N}$ of the approximating sequence $\mathcal{U}_{N}$ is tight on the Polish space $\Upsilon$. Hence,  by applying the Prohorov theorem and a variant of the Skorohod representation theorem (Theorem 1.10.4 in \cite{VW96}) we obtain the following convergence result.
\begin{thm}\label{skorohod} 
	There exists a probability space $(\bar\Omega,\bar\sF,\bar\bP)$ and random variables \\$\bar{\mathcal{U}}_{N}:=(
	\bar{\bu}^+_{N},\bar \bv^{+}_{N},\bar{\tilde\bfeta}^*_{N},\bar{\tilde\bfeta}_{N},m_N,k_N,\bar{\tilde\bU}_N,\bar{W}_N)$ and
	$\bar{\mathcal{U}}:=(
	\bar{\bu},{\bar \bv},\bar{\bfeta}^*,\bar{\bfeta},m,k, \bar{\tilde\bU},\bar{W})$
	defined on this new probability space, such that
	\begin{align}\label{convu}
		\bar{\mathcal{U}}_{N} \rightarrow \bar{\mathcal{U}},\qquad\quad\bar\bP-a.s. \quad\text{in the topology of } \Upsilon.
	\end{align}
	Additionally, there exist measurable maps $\phi_N:\bar\Omega\rightarrow\Omega$ such that \begin{align}\label{newrv}
		\bar{\mathcal{U}}_N(\bar\omega)={\mathcal{U}}_N(\phi_N(\bar\omega)) \quad \text{ for }\bar\omega\in\bar\Omega,
	\end{align} and $\bar \bP\circ\phi_N^{-1}=\bP$.
	
\end{thm}
Now we define rest of the approximate random variables:
	\begin{align*}
		\bar\bu_N=\bu_N\circ\phi_N,\quad 	\bar{\tilde\bu}_N=\tilde\bu_N\circ\phi_N, \quad \bar \bv_N=\bv_N\circ\phi_N\\
		\bar{\bfeta}_N={\bfeta}_N\circ\phi_N,\quad	{\bar{\bfeta}}^+_N={\bfeta}^+_N\circ\phi_N, \quad	\bar{\bfeta}^*_N={\bfeta}^*_N\circ\phi_N		.
	\end{align*}
	Then, from part (1), Lemmas \ref{difference} and an application of the Borel-Cantelli lemma, we obtain for $\frac32<s<2$ and $0<\beta<\frac12$ that
	\begin{align}
		\bar\bu_N \to \bar\bu,\quad 	\bar{\tilde\bu}_N \to \bar\bu, \quad \bar\bP-a.s. \quad\text{in } L^2(0,T;\bL^2(\sO)),\\	\bar{\bfeta}_N\to\bar{\bfeta},\quad{\bar{\bfeta}}^+_N\to\bar{\bfeta}, \quad	\bar{\bfeta}^*_N \to\bar{\bfeta}^*, \quad \bar\bP-a.s. \quad\text{in } L^2(0,T;\bH^s(0,L)),\\
		\bar \bv_N \to \bar \bv, \quad	\bar{\tilde \bv}_N \to \bar \bv, \quad \bar\bP-a.s. \quad\text{ in } L^2(0,T;\bH^{-\beta}(0,L)).
	\end{align}
		Thanks to these explicit maps we can identify the real-valued random variables ${k}_N= \|\bar \bv^*_N\|_{L^2(0,T;\bH_0^1(0,L))}$. Thus, a.s. convergence of ${k}_N$ implies that, for a fixed $\ep>0$, $\|\bar \bv^*_N\|_{L^2(0,T;\bH_0^1(0,L))}$ is bounded a.s. and thus, up to a subsequence, 
	\begin{align}\label{vweak}
		\bar\bv^*_N \rightharpoonup \bar\bv^* \quad\text{ weakly in } {L^2(0,T;\bH_0^1(0,L))} \quad \bar\bP-a.s.
	\end{align}
	Similarly $\bar\bu^+_N\rightharpoonup\bar\bu$ weakly in $L^2(0,T;V)$ a.s.
	
	Observe also that the bounds obtained in Lemma \ref{bounds} hold for the new random variables $\bar{\mathcal{U}}_{N}$ as well. Particularly, thanks to \eqref{convu} and \eqref{newrv}, we have the same deterministic bounds $\|\bar\bfeta^*\|_{C(0,T;\bH^s(0,L))} \leq \frac1{\delta_2}$ for  $\frac32<s<2$.
		We also have the following convergence results for the displacements. Namely,
	notice that Theorem \ref{skorohod} implies that for given $\frac32<s<2$ (see \cite{MC13} Lemma 3),
	\begin{align}\label{etaunif1}
		\bar{\bfeta}_{N} \rightarrow \bar{\bfeta} \text{ and } \bar{\bfeta}^*_{N} \rightarrow \bar{\bfeta}^* \,\,\,\text { in } L^\infty(0,T;\bH^s(0,L))\, a.s.
	\end{align}
	and thus,
	\begin{align}\label{etaunif}
		\bar{\bfeta}_{N} \rightarrow \bar{\bfeta} \text{ and } \bar{\bfeta}^*_{N} \rightarrow \bar{\bfeta}^*\,\,\, \text { in } L^\infty(0,T;\bC^1[0,L])\, a.s.
	\end{align}
Next we define piecewise constant interpolations of the ALE maps and Jacobians $A_{\bar\bfeta^*_N}, J_{\bar\bfeta^*_N}$, and their piecewise linear interpolation $\tilde{A}_{\bar\bfeta^*_N},\tilde { J}_{\bar\bfeta^*_N}$. Then \eqref{etaunif}, \eqref{boundJ} and \eqref{A2p}, yield
\begin{equation}
	\begin{split}\label{convrest}
		& \nabla A_{\bar\bfeta^*_N} \rightarrow \nabla A_{\bar\bfeta^*} \text{ and } (\nabla A_{\bar\bfeta^*_N})^{-1} \rightarrow (\nabla A_{\bar\bfeta^*})^{-1} \quad \text{ in } L^\infty(0,T;\bC(\sO))\, a.s.\\
		&			 J_{\bar\bfeta^*_N} \rightarrow  J_{\bar\bfeta^*} =\text{det}(\nabla A_{\bar\bfeta^*})\quad \text{ in } L^\infty(0,T;C(\bar\sO))\, a.s.\\
		&	 S_{\bar\bfeta_N} \rightarrow  S_{\bar\bfeta},\text{ and } S_{\bar\bfeta^*_N} \rightarrow  S_{\bar\bfeta^*}\quad \text{ in } L^\infty(0,T;C(\bar\Gamma)) \, a.s.
	\end{split}
\end{equation}	
Furthermore, $A_{\bar\bfeta^*} \in L^\infty(0,T;\bW^{2,p}(\sO)$ for $p<4$ and is the solution to \eqref{ale} corresponding to the boundary data $\text{\bf id} + \bar\bfeta^*$ on $\Gamma$.
Next let $\bar\bw^*_N=\partial_t \tilde{A}_{\bar\bfeta^*_N}=\sum_{n=0}^N\frac{1}{\Delta t}(A^\omega_{\bar{\bfeta}_*^{n}}-A^\omega_{\bar{\bfeta}^{n-1}_*})\chi_{(t_n,t_{n+1})}$. Note that for every $\omega\in\bar\Omega$, we have (see \cite{G11}):
\begin{align*}
	\|\bar\bw^*_N
	\|_{L^2(0,T;\bH^{k+\frac12}(\sO))}\leq C \|\bar\bv^{*}_N
	\|_{L^2(0,T;\bH^k(0,L))},\quad \text{ for any } 0\leq k\leq  1,
\end{align*}
where $C$ depends only on $k$. Thus using \eqref{vweak}, for a fixed $\ep>0$, we obtain that,
\begin{align}\label{convw}
	\bar\bw^*_N \rightharpoonup \bar\bw^*\, \quad \text{ weakly in } L^2(0,T;\bH^{1}(\sO)),\quad \, \bar\bP-a.s,
\end{align}
where $\bar\bw^*$ satisfies \eqref{eqn_w} with values $\bar\bv^*$ on $\Gamma$. Similarly, \eqref{boundJt2} gives us
\begin{align*}
	\partial_t \tilde J_{\bar\bfeta^*_N} \rightharpoonup 	\partial_t  J_{\bar\bfeta^*} \quad \text{ weakly in } L^2(0,T;\bL^{2}(\sO)),\quad \, \bar\bP-a.s.
\end{align*}

Finally, we give the definition of the required filtrations. we denote by $\sF_t'$ the $\sigma$-field generated by the random variables $(\bar{\bu}(s),\bar \bv(s)),\bar\bfeta(s),\bar{W}(s)$ for all $s \leq t$. Then we define
\begin{align}\label{Ft}
	\mathcal{N} &:=\{\mathcal{A}\in \bar{\mathcal{F}} \ |  \bar \bP(\mathcal{A})=0\},\qquad
	\bar{\mathcal{F}}^0_t:=\sigma(\mathcal{F}_t' \cup \mathcal{N}),\qquad
	\bar{\mathcal{F}}_t :=\bigcap_{s\ge t}\bar{\mathcal{F}}^0_s.
\end{align}
We note here that the stochastic processes $(J_{\bar\bfeta^*}\bar\bu,\bar \bfeta)$ that are $(\bar\sF_t)_{t\geq 0}-$progressively measurable.
For each $N$ we also define a filtration $\left(\bar{\mathcal{F}}^{N}_t \right)_{t \geq 0}$ on  $(\bar \Omega,\bar{\mathcal{F}},\bar \bP)$ the same way. 
  Moreover, using usual arguments we can see that $\bar W_N$ is an $\bar\sF^N_t$-Wiener process (see e.g. \cite{BFH18}).	
Next, relative to the new stochastic basis $(\bar\Omega,\bar\sF,(\bar\sF^{N}_t)_{t\geq0},\bar\bP,\bar W_N)$, 
thanks to \eqref{newrv} we can see that for each $N$, the following equation holds
\begin{equation}\begin{split}\label{approxsystem}	
		&(	(\tilde J_{\bar\bfeta^*_N}(t),\bq)+(\bar{\tilde \bv}_N(t),\bfpsi)
		=(J_0\bu_0,\bq)+(\bv_0,\bfpsi)-\int_0^t\langle \sL_e(\bar{\bfeta}^+_N),\bfpsi \rangle \\
		& +\int_0^t\int_{\sO}\frac{\partial\tilde J_{\bar\bfeta^*_N}}{\partial t}(2\bar{\tilde\bu}_N-{\bar \bu}_N)\cdot\bq-\frac12\frac{\partial\tilde J_{\bar\bfeta^*_N}}{\partial t}\bar\bu^{+}_N\cdot\bq -\ep\int_0^t\int_0^L\partial_{zz}\bar{\bv}^{\#}_N\cdot\partial_{zz}\bfpsi\\
		&-\frac12\int_0^t\int_{\sO}(J_{\bar\bfeta^*_N})((\bar\bu^+_N-
		\bar\bw^*_N)\cdot\nabla^{\bar{\bfeta}_N^*}\bar\bu^{+}_N\cdot\bq
		- (\bar\bu^+_N-
		\bar\bw^*_N)\cdot\nabla^{\bar{\bfeta}_N^*}\bq\cdot\bar\bu^{+}_N)\\
		&-2\nu\int_0^t\int_{\sO}(J_{\bar\bfeta^*_N})\bD^{\bar{\bfeta}_N^*}(\bar\bu^{+}_N)\cdot \bD^{\bar{\bfeta}_N^*}(\bq) -\frac1\alpha\int_0^t\int_\Gamma  S_{\bar\bfeta_N}(\bar\bu^{+}_N-\bar\bv^{+}_N)(\bq-\bfpsi)\\
		&- \frac1\ep\int_0^t\int_\sO 
		\text{div}^{{\bar\bfeta}^*_N}\bar\bu^{+}_N\text{div}^{{\bar\bfeta}^*_N}\bq d\bx -\frac1\ep\int_0^t\int_{\Gamma} (\bar\bu^{+}_N-\bar\bv^{+}_N)\cdot\bar\bn^*_N(\bq-\bfpsi^{})\cdot\bar\bn^*_N\\
		&  
	\pm	\int_0^t\left( P_{{in/out}}\int_{0}^1q_z\Big|_{z=0/1}\right)
		+\int_0^t
		(G({\bar \bu}_N,\bar{\bfeta}^*_N)d\bar W_N, \bQ)+(E_N(t),\bQ),
\end{split}\end{equation}
$\bar\bP$-a.s. for every $t \in [0,T]$ and any $\bQ\in \sD\cap \sU_1$.

Using the convergence results stated in Theorem \ref{skorohod} we can then pass $N \to \infty$ in the deterministic terms in \eqref{approxsystem}.
For the stochastic integral see Lemma \ref{conv_G}. We mention here how we treat the convective term which is the only term that needs an explanation. By integrating by parts we obtain
\begin{align}
\notag	\int_\sO  J_{\bar\bfeta^*_N} \bar\bw_N^*\cdot\nabla^{\bar\bfeta_N^*}\bar\bu^+_N\cdot\bq&=-\int_{\sO} {J}_{\bar\bfeta^*_N} \nabla^{\bar\bfeta^*_N}\cdot\bar\bw_N^*\bar{\bu}^+_N \cdot\bq-\int_{\sO} {J}_{\bar\bfeta^*_N} \bar\bw^*_N\cdot \nabla^{\bar\bfeta^*_N}\bq \cdot \bar{\bu}^+_N\\
	&+\int_\Gamma S_{\bar\bfeta_N^*}(\bar\bv_N^*\cdot\bn_N^*)(\bar\bu^+_N\cdot\bq),\label{byparts}
\end{align}
where $S_{\bar\bfeta^*_N}$ is the Jacobian of the transformation from Eulerian to Lagrangian coordinates.
Thus, using the weak and strong convergence results in Theorem \ref{skorohod}, \eqref{vweak}, \eqref{convw} and \eqref{etaunif} we can pass $N\to \infty$ in \eqref{byparts}. Then using the fact that,
\begin{align}\label{Jform}	\int_0^t\int_{\sO}\partial_tJ_{\bar\bfeta^*}\bar{\bu} \cdot\bq=\int_0^t\int_{\sO}J_{\bar\bfeta^*} \nabla^{\bfeta^*}\cdot\bw^*\bar{\bu} \cdot\bq,
\end{align}
we arrive at the following weak formulation.
	\begin{thm}\label{exist1} 
		For the stochastic basis $(\bar\Omega,\bar\sF,(\bar\sF_t)_{t \geq 0},\bar\bP,\bar W)$ constructed in Theorem \ref{skorohod}, given any fixed $\ep>0$ and $\delta=(\delta_1,\delta_2)$ satisfying \eqref{delta},  {the processes $(\bar\bu,\bar\bfeta,\bar\bfeta^*)$ obtained in Theorem \ref{skorohod} are such that} $\bar\bfeta^*$ and $(\bar J_{\bfeta^*}\bar\bu,\partial_t\bar\bfeta)$ are $(\bar\sF_t)_{t \geq 0}$-progressively measurable with $\bar\bP$-a.s. continuous paths in $\bH^s(0,L)$, $s<2$ and $\sV_1'$ respectively and the following weak formulation 
		holds $\bar\bP$-a.s. for every $t\in[0,T]$ and $\bQ\in \sD$:
		\begin{equation}\begin{split}\label{martingale1}
			&(	{J}_{\bar\bfeta^*}(t)\bar\bu (t),\bq)+(\partial_t\bar{ {\bfeta}}(t),\bfpsi)
			=((J_0)\bu_0,\bq)+( \bv_0,\bfpsi)-\int_0^t \langle \sL_e(\bar{\bfeta}),\bfpsi \rangle \\
			&-\frac12\int_0^t\int_{\sO}(J_{\bar\bfeta^*})(\bar\bu\cdot\nabla^{\bar\bfeta ^*}\bar\bu \cdot\bq
			- (\bar\bu^{}-2\bar\bw^*)\cdot\nabla^{\bar\bfeta ^*}\bq\cdot\bar\bu ) 		\\
			&+\frac12\int_0^t\int_\Gamma S_{\bar\bfeta^*}(\partial_t\bar{ {\bfeta}}^*\cdot\bn^*)(\bar\bu\cdot\bq)	-\frac1\alpha\int_0^t\int_\Gamma S_{\bar\bfeta} (\bar\bu-\partial_t\bar{\bfeta})(\bq-\bfpsi) \\
			&-2\nu\int_0^t\int_{\sO}J_{\bar\bfeta^*}\bD^{{\bar\bfeta}^*}(\bar\bu)\cdot \bD^{{\bar\bfeta}^*}(\bq) d\bx -\ep \int_0^t\int_{0}^L\partial_{zz}\partial_t\bar{ {\bfeta}}\cdot\partial_{zz}\bfpsi dz \\
			&- \frac{1}{\ep}\int_0^t\int_{\sO}
			\text{div}^{{\bar\bfeta}^*}\bar\bu\text{div}^{{\bar\bfeta}^*}\bq d\bx -\frac{1}{\ep}\int_0^t\int_{\Gamma} (\bar\bu-\partial_t\bar{ {\bfeta}})\cdot\bn^*(\bq-\bfpsi)\cdot\bn^*\\
			&+ 
			\int_0^t\left( P_{{in}}\int_{0}^1q_z\Big|_{z=0}dr-P_{{out}}\int_{0}^1q_z\Big|_{z=1}dr\right)
			ds +\int_0^t
			(	G(\bar\bu ,\bar{\bfeta}^*)d\bar W , \bQ) .
		\end{split}
	\end{equation}
	\end{thm}
	{\s Next, we argue that 
		\begin{align}\label{etasequal1}
			\bar{\bfeta}^*(t)=\bar{\bfeta}(t) \quad \text{ for any } t<T^{\bfeta}, \quad \bar\bP-a.s.
		\end{align}	
		{ where for a given $\delta=(\delta_1,\delta_2)$,
			$$	T^{\bfeta}  := T\wedge\inf\{t>0:\inf_{\sO}(J_{\bar\bfeta}(t))\leq \delta_1 \text{ or } \|\bar{\bfeta}(t)\|_{\bH^s(0,L)}\geq \frac1{\delta_2} \}.$$}
		Indeed, to prove \eqref{etasequal1}, we introduce the following stopping times.
		For $\frac32<s<2$ we define
		\begin{align*}
			T^{\bfeta}_N &:=T\wedge \inf\{t> 0:\inf_{\sO}(J_{\bar\bfeta_N}(t))\leq \delta_1 \text{ or } \|\bar{\bfeta}_N(t)\|_{\bH^s(0,L)}\geq \frac1{\delta_2}\}.
		\end{align*}
		Then  \eqref{etaunif} implies that $T^{\bfeta} \leq \liminf_{N\rightarrow\infty}T^{\bfeta}_N$ a.s. 
		Observe further that for almost any $\omega\in\bar\Omega$ and $t<T^{\bfeta}$, and for any $\epsilon>0$, there exists an $N$ such that
		\begin{align*}
			\|\bar{\bfeta}(t)-\bar{\bfeta}^*(t)\|_{\bH^s(0,L)}&<\|\bar{\bfeta}(t)-\bar{\bfeta}_N(t)\|_{\bH^s(0,L)}+\|\bar{\bfeta}^*(t)-\bar{\bfeta}^*_N(t)\|_{\bH^s(0,L)}\\
			&+\|\bar{\bfeta}^*_N(t)-\bar{\bfeta}_N(t)\|_{\bH^s(0,L)}
			<\epsilon.
		\end{align*}
		This is true because, the uniform convergence \eqref{etaunif} implies that for any $\epsilon>0$ there exists an $N_1 \in \mathbb{N}$ such that the first two terms on the right side of the above inequality
		are each bounded by $\frac{\epsilon}2$ for all $N\geq N_1$.
		Moreover, since $t<T^{\bfeta}_N$ for infinitely many $N$'s, the third term is equal to 0. This concludes the proof of \eqref{etasequal1}.
	}

	\section{Passing to the limit $\ep\to 0$}
To emphasize the dependence on the parameter $\ep>0$, hereon we will denote the martingale solution found in the previous section by $(\bar \bu_\ep,\bar \bv_\ep,\bar \bv^*_\ep, \bar \bfeta_\ep,\bar\bfeta^*_\ep,\bar W_\ep)$ and $(\Omega^\ep,\sF^\ep,(\sF^\ep_t)_{t\geq 0},\bP^\ep)$.  
The aim of this section is to pass $\ep \rightarrow 0$ in \eqref{martingale1} by constructing appropriate test functions. Most of the results in the first half of this section can be proven as in the previous section and so we will only summarize the important theorems without proof. Observe that, thanks to the weak lower-semicontinuity of norm, the uniform estimates obtained in the previous section still hold. As a consequence of Lemmas \ref{bounds} and Theorem \ref{skorohod}, we thus have the following uniform bounds. 
\begin{lem}[Uniform boundedness]\label{boundsep}
	For a fixed $\delta=(\delta_1,\delta_2)$ that satisfies \eqref{delta}, we have for some $C>0$ {\bf independent of} $\ep$ that
	\begin{enumerate}
		\item $\bE^\ep\|\bar\bu_\ep\|^2_{L^\infty(0,T;\bL^2(\sO))\cap L^2(0,T;V)}<C.$
		\item $\bE^\ep\|\bar \bv^*_\ep\|^2_{L^\infty(0,T;\bL^2(0,L))}<C$.
		\item $\bE^\ep\|\bar \bfeta_\ep\|^2_{L^\infty(0,T;\bH^2_0(0,L)\,\cap\, W^{1,\infty}(0,T;\bL^2(0,L)))}<C.$
		\item $\bE^\ep\|\bar \bfeta^*_\ep\|^2_{L^\infty(0,T;\bH^2_0(0,L)\,\cap\, W^{1,\infty}(0,T;\bL^2(0,L)))}<C.$
		\item $\bE^\ep\|\text{div}^{\bar\bfeta^*_\ep}\bar\bu_\ep\|^2_{L^2(0,T;L^2(\sO))} < C{\ep}$, $\bE^\ep\|(\bar\bu_\ep|_\Gamma-\bar\bv_\ep)\cdot\bn^*_\ep\|^2_{L^2(0,T;L^2(0,L))} < C{\ep}$.
		\item $\sqrt{\ep}\bE^\ep\| \partial_{zz}\bar\bv_\ep\|^2_{L^2(0,T;\bL^2(0,L))} < C$.
		\item $\|\bar\bfeta_\ep^*\|_{C(0,T;\bH^s(0,L))} \leq \frac{1}{\delta_2}$ for  $\frac32<s<2$, for almost every $\omega\in\Omega^\ep$.
	\end{enumerate}
\end{lem}
	Next we have the following tightness results:
\begin{lem}[Tightness of the laws]\label{tightep}
	\begin{enumerate}
		\item The laws of $\bar\bu_\ep$ and $\bar \bv_\ep$ are tight in $L^2(0,T;\bH^\alpha(\sO))$ and $ L^2(0,T;\bH^{-{\beta}}(0,L))$ respectively for any $0\leq \alpha<1$, $0<\beta<\frac12$.
		\item The laws of $\bar \bfeta_\ep$ and that of $\bar \bfeta^*_\ep$ are tight in $C([0,T];\bH^s(0,L))$ for $\frac32<s<2$.
		\item The laws of $\|\bar \bu_\ep\|_{L^2(0,T;V)}$ are tight in $\mathbb{R}$.
		\item The laws of $\|\bar \bv^*_\ep\|_{L^2(0,T;\bL^2(0,L))}$ are tight in $\mathbb{R}$.
	\end{enumerate}
\end{lem}
\begin{proof}
	We describe how to prove the first statement, which follows from the proof of Lemma \ref{tightuv} almost identically. Construction of suitable test functions $(\bq_\ep,\bfpsi_\ep) $ is the same as \eqref{Qn} and we apply the variant of It\^o's formula stated in 
	Lemma 5.1 in \cite{BO13}, which justifies testing \eqref{martingale1} with the continuous-in-time versions of the random test functions \eqref{Qn} (i.e. where $(\Delta t)\sum_{n-j+1}^n$ is replaced by $\int_{t-h}^t dt$). Recall that all the bounds obtained in the proof of Lemma \ref{tightuv}, except in \eqref{ep2}, do not depend on $\ep$. However, due to integrating by parts in \eqref{byparts} and applying \eqref{Jform}, the weak formulation \eqref{martingale1} now contains the boundary integral $\int_0^t\int_\Gamma S_{\bar\bfeta_\ep^*}(\bar\bv_\ep^*\cdot\bn_\ep^*)(\bar\bu_\ep\cdot\bq_{\ep})$ instead of  the aforementioned term involving the derivatives of $\bar\bw_\ep^*$. Then, for the process $\bq_\ep$ taking values in $\bH^1(\sO)$, described above and constructed as in \eqref{Qn}, we can bound this boundary integral independently of $\ep$, by using the fact that $\|\bn^*_\ep\|_{\bL^\infty((0,T)\times(0,L))}<C(\delta)$ together with the bounds for the trace $\bar\bu_\ep|_{\Gamma}$ in $L^2(\hat\Omega;L^2(0,T;\bH^\frac12(\Gamma)))$ (see Theorem 1.5.2.1 in \cite{G11}) and that for $\bar\bv^*_\ep$ in $L^2(\hat\Omega;L^2(0,T;\bL^2(0,L)))$ which are independent of $\ep$.
\end{proof}
	Now for an infinite denumerable set of indices $\Lambda$, we denote by 
	$\mu_{\ep}$ the joint law of the random variable $\bar{\mathcal{U}}_{\ep}:=(
	\bar{\bu}_{\ep},\bar \bv_{\ep},\bar\bfeta_{\ep}, \bar\bfeta^*_{\ep},
	\|\bar \bu_\ep\|_{L^2(0,T;V)},\|\bar \bv^*_\ep\|_{L^2(0,T;\bL^2(0,L))},
	\bar{W}_\ep)$
	taking values in the phase space
	\begin{align*}
		\sS&:=L^2(0,T;\bH^{\frac34}(\sO))\times L^2(0,T; \bH^{-\beta}(0,L)) \times[ C([0,T],\bH^s(0,L))]^2
		\times \mathbb{R}^2\times
		C([0,T];U),
	\end{align*}
	for some $0<\beta<\frac12$, $\frac32<s<2$.
	
	Then the tightness of the laws $\mu_\ep$ on $\sS$ and an application of the Prohorov theorem and the almost sure representation in \cite{VW96} give us the following result.
	\begin{thm}\label{skorohod2} There exists a probability space $(\hat\Omega,\hat\sF,\hat\bP)$ and random variables \\ $\hat{\mathcal{U}}_{\ep}=(
		\hat{\bu}_{\ep},\hat \bv_{\ep},\hat\bfeta_{\ep}, \hat\bfeta^*_{\ep}, m_\ep, k_\ep,
		\hat{W}_\ep)$ and
		$\hat{\mathcal{U}}=(
		\hat{\bu},\hat \bv,\hat\bfeta, \hat\bfeta^*,{m},{k},
		\hat{W})$
		such that
		\begin{enumerate}
			\item $\hat{\mathcal{U}}_{\ep}=^d\bar{\mathcal{U}}_{\ep}$ for every $\ep \in \Lambda$.
			\item $\hat{\mathcal{U}}_{\ep} \rightarrow \hat{\mathcal{U}}$ $\hat\bP$-a.s. in the topology of $\sS$ as $\ep\rightarrow 0$.
			\item $\partial_t\hat\bfeta=\hat \bv$ and $\partial_t\hat\bfeta^*=\hat \bv^*$, in the sense of distributions, almost surely.
		\end{enumerate}
	\end{thm}
	
	We recall again that Theorem 1.10.4 in \cite{VW96} tells us that the random variables $\hat{\mathcal{U}}_\ep$ can be chosen such that for every $\ep \in \Lambda$,
	\begin{align}\label{newrv2}
		\hat{\mathcal{U}}_\ep(\omega)=\bar{\mathcal{U}}_\ep(\phi_\ep(\omega)), \quad \omega \in \hat\Omega,
	\end{align} and $\hat\bP\circ\phi_\ep^{-1}=\bP^\ep$, where $\phi_\ep:\hat\Omega\rightarrow \Omega^\ep$ is measurable. 
	
	Thanks to these explicit maps we identify the real-valued random variables $m_\ep$ as $m_\ep= \|\hat \bu_\ep\|_{L^2(0,T;V)}$ and notice that $m_\ep$ converge almost surely due to  Theorem~\ref{skorohod2}.}
Hence as in Theorem \ref{skorohod}, we obtain, up to a subsequence, that
	\begin{align}\label{uweak2}
	\hat \bu_\ep \rightharpoonup \hat\bu \quad\text{ weakly in } {L^2(0,T;V)} \quad \hat\bP-a.s.
	\end{align}
Similarly,
\begin{align}\label{vweak2}
\hat \bv^*_\ep \rightharpoonup \hat \bv^* \quad\text{ weakly in } {L^2(0,T;\bL^2(0,L))} \quad \hat\bP-a.s.
\end{align}
As in the previous section, we also have that
	\begin{align}\label{etaunif2}
	\hat{\bfeta}_{\ep} \rightarrow \hat{\bfeta} \text{ and } \hat{\bfeta}^*_{\ep} \rightarrow \hat{\bfeta}^*\,\,\, \text { in } L^\infty(0,T;\bC^1[0,L])\, a.s.,
\end{align}
and that,
\begin{equation}\label{convrest1}
\begin{split}
	&  A_{\hat\bfeta^*_\ep} \rightarrow  A_{\hat\bfeta^*}, ( A_{\hat\bfeta^*_\ep})^{-1} \rightarrow ( A_{\hat\bfeta^*})^{-1}  \text{ in } L^\infty(0,T;\bW^{2,p}(\sO))\, \text{ for any $p<4$}\, \, \, a.s.,\\
	&			 J_{\hat\bfeta^*_\ep} \rightarrow  J_{\hat\bfeta^*} =\text{det}(\nabla A_{\hat\bfeta^*})\quad \text{ in } L^\infty(0,T;C(\bar\sO))\, 
	 \quad \, \hat\bP-a.s,
	\\	& S_{\hat\bfeta_\ep} \rightarrow  S_{\hat\bfeta}\quad \text{ in } L^\infty(0,T;C(\bar\Gamma)) \, \quad \, \hat\bP-a.s,\\
	&	\hat\bw^*_\ep \rightharpoonup \hat\bw^*\, \quad \text{ weakly in } L^2(0,T;\bH^{\frac12}(\sO)),\quad \, \hat\bP-a.s.\\
		&\hat	\bn^*_\ep \to \hat\bn^*\, \quad \text{  in } L^\infty(0,T;C(\bar\Gamma)),\quad \, \hat\bP-a.s.
\end{split}
\end{equation}
Due to the lack of the equivalent of Lemma \ref{uvtightC}, we have one more obstacle to deal with. Namely,
that the candidate solution for fluid velocity, $\hat\bu$, does not have the desired temporal regularity to be a stochastic process in the classical sense. Hence, we construct an appropriate filtration as follows: first
define the $\sigma-$fields
$$
\sigma_t(\hat\bu):=\bigcap_{s\geq t}\sigma\left(\bigcup_{\bQ\in C^\infty_0((0,s);\sD)}\{(\hat\bu,\bq)<1\}\cup \mathcal{N} \right),\quad \mathcal{N}=\{\mathcal{A}\in \hat{\mathcal{F}} \ |  \hat\bP(\mathcal{A})=0\}.
$$
Let $\hat\sF_t'$ be the $\sigma-$ field generated by the random variables $\hat\bfeta(s),\hat{W}(s)$ for all $0\leq s \leq t$.
Then we define the history of the random distributions $\hat\bu$, $\hat{\mathcal{F}}_t$, as follows
\begin{align}\label{Ft1}
\hat{\mathcal{F}}^0_t:=\bigcap_{s\ge t}\sigma(\hat{\sF}_s' \cup \mathcal{N}),\qquad
\hat{\mathcal{F}}_t :=\sigma(\sigma_t(\hat\bu)\cup \hat{\mathcal{F}}^0_t).
\end{align}
This gives a complete, right-continuous filtration $(\hat{\mathcal{F}}_t)_{t \geq 0}$, on the probability space $(\hat\Omega,\hat{\mathcal{F}},\hat\bP)$, to which the noise and the candidate solution $\hat\bu$ are adapted. Now we state the following result from \cite{BFH18}.
\begin{lem}\label{representative}
There exists a stochastic process taking values in $
L^2(0,T;\bL^2(\sO))$ a.s. which is an $(\hat\sF_t)_{t\geq 0}$-progressively measurable representative of $\hat\bu$.
\end{lem}
	
		\begin{thm}\label{exist2}
		For any fixed $\delta=(\delta_1,\delta_2)$ that satisfies \eqref{delta}, the randome variables $(\hat \bu,\hat {\bfeta},\hat{\bfeta}^*)$ constructed in Theorem~\ref{skorohod2}
		satisfy the following
		\begin{equation}\begin{split}\label{martingale2}
				&(	{J}_{\hat\bfeta^*}(t)\hat\bu (t),\bq(t))+(\partial_t\hat{ {\bfeta}}(t),\bfpsi(t))
				=((J_0)\bu_0,\bq(0))+( \bv_0,\bfpsi(0))\\
				&+\int_0^t\int_\sO{ J}_{\hat\bfeta^*}\hat\bu\cdot \partial_t{\bq} +\int_0^t\int_0^L \partial_t\hat {\bfeta} \,\partial_t{\bfpsi}-\int_0^t \langle \sL_e(\hat{\bfeta}),\bfpsi \rangle \\
				&-\frac12\int_0^t\int_{\sO}J_{\hat\bfeta^*}(\hat\bu\cdot\nabla^{\hat\bfeta ^*}\hat\bu \cdot\bq
			- (\hat\bu^{}-2\hat\bw^*)\cdot\nabla^{\hat\bfeta ^*}\bq\cdot\hat\bu ) 		+\frac12\int_0^t\int_\Gamma S_{\hat\bfeta^*}(\hat\bv^*\cdot\hat\bn^*)(\hat\bu\cdot\bq)\\
				&-2\nu\int_0^t\int_{\sO}{J}_{\hat\bfeta^*}\, \bD^{\hat{\bfeta} ^*}(\hat\bu )\cdot \bD^{\hat{\bfeta} ^*}(\bq) -\frac1\alpha\int_0^t\int_\Gamma S_{\hat\bfeta} (\hat\bu-\partial_t\hat{\bfeta})\cdot\tau^{\hat\bfeta}(\bq-\bfpsi)\cdot\tau^{\hat\bfeta}  \\
				&+ 
				\int_0^t\left( P_{{in}}\int_{0}^1q_z\Big|_{z=0}dr-P_{{out}}\int_{0}^1q_z\Big|_{z=1}dr\right)
				ds +\int_0^t
				(	G(\hat\bu ,\hat{\bfeta}^*)d\hat W , \bQ)  ,
			\end{split}
		\end{equation}
		$\hat\bP$-a.s. for almost every $t\in[0,T]$ and for any $(\hat\sF_t)_{t \geq 0}-$adapted process $\bQ=(\bq,\bfpsi)$ with $C^1$- paths in $\sD$ such that $\nabla^{\hat{\bfeta}^*}\cdot\bq=0$ and $\bq|_{\Gamma}\cdot\bn^{\hat\bfeta^*}=\bfpsi\cdot\bn^{\hat\bfeta^*}$ a.s. Moreover, $\nabla^{\hat{\bfeta}^*}\cdot\hat\bu=0$.
	\end{thm}
	
	\begin{proof}[Proof of Theorem \ref{exist2}]
First we must construct $\sD$-valued test processes $(\bq_{\ep},\bfpsi_{\ep})$, satisfying the kinematic coupling condition and such that $\bq_\ep$ satisfies the transformed divergence-free condition. This is required so that the two penalty terms in the approximate weak formulation drop out. 
		
		We first construct an appropriate test functions for the limiting equation \eqref{martingale2} as follows: Recall that the maximal domain $\sO_\delta=(0,L)\times(0,R_\delta)$ is a rectangular domain comprising of all the moving domains $\sO_{\hat\bfeta_\ep^*}$. Consider a smooth $(\hat\sF_t)_{t\geq0}$-adapted process $\bg=(g_z,g_r)$ on $\bar\sO_{\delta}$ such that $\nabla\cdot \bg=0$ and such that $\bg$ satisfies the required boundary conditions  $g_r=0 \text{ on } z=0,L, r=0$  and $\partial_r g_z=0 \text{ on }\Gamma_{b}$. Assume also that, on the top lateral boundary of the moving domain associated with $\hat{\bfeta}^*$, $\Gamma_{\hat{\bfeta}^*}$, the function $\bg$ satisfies
		$\bg(t)|_{\Gamma_{\hat{\bfeta}^*(t)}}\cdot\hat\bn^*(t)=\bfpsi(t)\cdot\hat\bn^*(t)$ for some smooth $(\hat\sF_t)_{t\geq0}$-adapted process $\bfpsi=(\psi_z,\psi_r)$.
	We define
		$$ \bq(t,z,r,\omega) = \bg(t,\omega) \circ\, A^\omega_{\hat{\bfeta}^*}(t)(z,r) .$$	
		To observe that $\bq$ is a suitable test function we consider, for any $t\in[0,T]$ and given process $\bg$, the map $\sC_\bg:\hat\Omega\times \bC([0,L]) \rightarrow \bC^1(\bar\sO)$, $$\sC_{\bg}(\omega,{\bfeta})
		=F_{\bfeta}(\bg(t,\omega)),$$
		where $F_{\bfeta}({\bf f}):={\bf f}\circ A^\omega_{{\bfeta}}$ is a well-defined map from $\bC(\bar\sO_{{\bfeta}})$ to $\bC(\bar\sO)$ for  any ${\bfeta}\in \bC([0,L])$. Due to the continuity of the composition operator $F_{\bfeta}$, the assumption that $\bg(t)$ is $\hat\sF_t-$measurable, implies for any ${\bfeta}$, that the $\bC^1(\bar\sO)$-valued map $\omega \mapsto \sC_\bg(\omega,{\bfeta})$ is $\hat\sF_t$-measurable (where $\bC^1(\bar\sO)$ is endowed with Borel $\sigma$-algebra).
		Note also that for a fixed $\omega$, the map  ${\bfeta}\mapsto \sC_\bg(\omega,{\bfeta})$ is continuous.
		Hence, we infer that $\sC_\bg$ is a Carath\'eodory function.
		Recall also that $\hat{\bfeta}^*$ is $(\hat\sF_t)_{t\geq 0}$-adapted. Therefore, we deduce that the $\bC^1(\bar\sO)$-valued process $\bq(t,\omega)=\sC_\bg(\omega,\hat{\bfeta}^*(t,\omega))$ is $(\hat\sF_t)_{t\geq 0}$-adapted as well. 

In summary, we have $\{\hat\sF_t\}_{t \geq 0}$-adapted processes $(\bq,\bfpsi)$ with continuous paths in $\sD$ such that $$\nabla^{\hat\bfeta^*}\cdot\bq=0 \text{ and } \bq\Big|_{\Gamma}\cdot\bn^{\hat\bfeta^*}=\bfpsi\cdot\bn^{\hat\bfeta^*}.$$
Moreover, for any $\omega\in\hat\Omega$ that $\bq\in L^\infty(0,T;\bH^{2+k}(\sO)) \cap H^1(0,T;\bH^k(\sO))$ for any $k\leq \frac12$.
Now we define the approximate test functions $(\bq_{\ep},\bfpsi_{\ep})$, with the aid of the Piola transformation as done in the proof of Lemma \ref{tightuv}:
\begin{align*}
&\bq_{\ep}={J^{-1}_{\hat\bfeta^*_\ep}}\nabla A_{\hat\bfeta^*_{\ep}}{J^{}_{\hat\bfeta^*}} \nabla  A^{-1}_{\hat\bfeta^*} \left( \bq-\bfpsi\chi\right) +\bfpsi\chi-\frac{\lambda^{\hat{\bfeta}^*_{\ep}}-\lambda^{\hat\bfeta^*}}{\lambda^{\ep}_0}(\xi_0\chi)\\
&+{J^{-1}_{\hat\bfeta^*_\ep}}\sB\left( \text{div}\left( (J_{\hat\bfeta^*}(\nabla A_{\hat\bfeta^*})^{-1}-J_{\hat\bfeta_\ep^*}(\nabla A_{\hat\bfeta_\ep^*})^{-1})\bfpsi\chi-\frac{\lambda^{\hat{\bfeta}^*_{\ep}}-\lambda^{\hat\bfeta^*}}{\lambda^{\ep}_0}J_{\hat\bfeta_\ep^*}(\nabla A_{\hat\bfeta_\ep^*})^{-1}\xi_0\chi\right) \right)
\end{align*}	
and,
\begin{align*}	
&	\bfpsi_{\ep}= \bfpsi-\frac{\lambda^{\hat{\bfeta}^*_{\ep}}-\lambda^{\hat{\bfeta}^*}}{\lambda^{\ep}_0}\xi_0,
\end{align*}	
where we pick an appropriate $\xi_0 \in \bC_0^\infty(\Gamma)$ such that $\lambda_0^\ep$ defined below is not 0 for any $\ep>0$, 
$$\lambda_0^{\ep}(t)=-\int_\Gamma ({\bf id}+\hat{\bfeta}^*_{\ep}(t))\times \partial_z\xi_0\,dz. \quad\text{}$$
We also define the real-valued corrector functions,
	$$\lambda^{\hat{\bfeta}^*_{\ep}}(t)=-\int_\Gamma ({\bf id} +\hat{\bfeta}_{\ep}^*(t))\times \partial_z\bfpsi (t)dz,\quad \lambda^{\hat{\bfeta}^*_{}}(t)=-\int_\Gamma ({\bf id} +\hat{\bfeta}_{}^*(t))\times \partial_z\bfpsi (t)dz.$$
As earlier, $\chi(r)$ is a smooth function on $\sO$ such that $\chi(1)=1$ and $\chi(0)=0$.
	Observe that the properties of the Piola transformation (see e.g. Theorem 1.7 in \cite{C88}), imply that
$$\nabla^{\hat\bfeta^*_\ep} \cdot\bq_{\ep}=J^{}_{\hat\bfeta^*}J^{-1}_{\hat\bfeta_\ep^*} \nabla^{\hat\bfeta^*}\cdot\bq=0, \qquad \text{and}\qquad \bq_\ep|_\Gamma\cdot\bn_\ep^*=\bfpsi_\ep\cdot\bn_\ep^*.
$$
Furthermore, we have
$$|\frac{d}{dt}\lambda^{\hat{\bfeta}^*_{\ep}}|
\leq \|\hat\bv^*_\ep\|_{\bL^{2}(0,L)}\|\bfpsi\|_{\bH_0^{1}(0,L)}+\|\hat\bfeta^*_\ep\|_{\bL^\infty(0,L)}\|\partial_{t}\partial_z\bfpsi\|_{\bL^\infty(0,L)}.$$
 Hence $\lambda^{\hat{\bfeta}^*_{\ep}}\to \lambda^{\hat{\bfeta}^*}$ strongly in $L^\infty(0,T)$ and weakly 
 in $H^1(0,T)$ a.s.  
 Additionally, thanks to \eqref{convrest1} 
 we obtain that
 \begin{align}
 \notag&\|	\bq_\ep-\bq\|_{L^\infty(0,T;\bH^1(\sO))}\\ \notag&\leq \|A_{\hat\bfeta^*_\ep}-A_{\hat\bfeta^*}\|_{L^\infty(0,T;\bW^{2,3}(\sO))}\left( \|\bq\|_{L^\infty(0,T;\bH^1(\sO))}+\|\bfpsi\|_{L^\infty(0,T;\bH^1_0(0,L))}\right) \\
 &\qquad+\|\lambda^{\hat{\bfeta}^*_{\ep}}-\lambda^{\hat{\bfeta}^*}\|_{L^\infty(0,T)}\to 0 \qquad \hat\bP-\text{ a.s.}\label{convq}
 \end{align}
Furthermore, for any $1<m<\frac32$ we can see that (see e.g. \cite{BH21}, Theorem 1.4.4.2 \cite{G11}), 
 \begin{align*}
 	\|\partial_t\bq_\ep\|_{\bH^{-\frac12}(\sO)} \leq C(\delta)\|\bw^*_\ep\|_{\bH^{\frac12}(\sO)}\| \nabla A^{}_{\hat\bfeta^*}\|_{\bH^{m}(\sO)}\|\bq\|_{\bH^m(\sO)} + C(\delta)\|\partial_t\bq\|_{\bL^2(\sO)}\leq C.
 \end{align*}
 Hence, we deduce for any $\frac12\leq\alpha\le 1$ that,
 \begin{align}\label{convqt}
 	\partial_t\bq_\ep \rightharpoonup \partial_t\bq
 	\quad\text{ weakly in } L^2(0,T;\bH^{-\alpha}(\sO)), \quad \hat\bP-a.s. \, 
 \end{align}
Similarly, for any $k$ we have 
	\begin{equation}\begin{split}\label{convpsi}
			\bfpsi_{\ep} \rightarrow \bfpsi \quad \text{ in }L^\infty(0,T;\bC^k(\bar\Gamma)),\, \quad \hat\bP-a.s.\\
			\partial_t\bfpsi_{\ep} \rightharpoonup \partial_t \bfpsi \quad \text{ weakly in }L^2(0,T;\bC^k(\Gamma)),\, \quad \hat\bP-a.s.
		\end{split}
	\end{equation}

	Now we test \eqref{martingale1} with $(\bq_\ep,\bfpsi_\ep)$ for which we invoke the variant of It\^o's formula derived in Lemma 5.1 in \cite{BO13}. We can now pass $\ep\to 0$ starting with the stochastic integral.
		\begin{lem}\label{conv_G}
			The processes $
			\left( \int_0^t(G(\hat{\bu}_{\ep}(s),\hat{\bfeta}^*_{\ep}(s))d\hat{W}_{\ep}(s),\bQ_{\ep}(s))\right) _{t \in [0,T]}$ converge \, to \\ $\left( \int_0^t(G(\hat{\bu}(s),\hat{\bfeta}^*(s))d\hat{W}(s),\bQ(s))\right) _{t \in [0,T]}$  in $L^1(\hat{ \Omega};L^1(0,T;\mathbb{R}
			))$ as $ \ep \rightarrow 0$.
		\end{lem}
\begin{proof}
Under the assumptions \eqref{growthG} we observe have, 
\begin{align*}
&\int_0^T\|(G(\hat{\bu}_{\ep},\hat{\bfeta}^*_{\ep}),\bQ_{\ep})-( G(\hat \bu,\hat{\bfeta}^*),\bQ)\|^2_{L_2(U_0,\R)}d s\\
&\leq \int_0^T \|(G(\hat{\bu}_{\ep},\hat{{\bfeta}}^*_{\ep})-G(\hat{\bu},\hat{{\bfeta}}^*),\bQ_{\ep})\|^2_{L_2(U_0;\mathbb{R})} + \int_0^T\|(G(\hat{\bu},\hat{{\bfeta}}^*_{}),\bQ_{\ep}-\bQ)\|^2_{L_2(U_0;\mathbb{R})}\\
&\le \int_0^T \left( \|\hat{\bfeta}^*_{\ep}-\hat{\bfeta}^*\|^2_{\bL^2(0,L)}+\|\hat{\bu}_{\ep}-\hat\bu\|^2_{\bL^2(\sO)}\right)  \|\bQ_{\ep}\|_{\bL^2}^2\\
&+  \int_0^T \left( \|\hat{\bfeta}^*\|^2_{\bH^2_0(0,L)}+\|\hat\bu\|^2_{\bL^2(\sO)}\right)  \|\bQ_{\ep}-\bQ\|_{\bL^2}^2.
\end{align*}	
Then thanks to Theorem \ref{skorohod}, \eqref{convq} and \eqref{convpsi}$_1$ the right side of the inequality above converges to 0, $\hat\bP$-a.s. as $\ep \to 0$.	That is,
\begin{align}\label{g1}
	(G(\hat{\bu}_{\ep},\hat{\bfeta}^*_{\ep}),\bQ_{\ep})\rightarrow( G(\hat \bu,\hat{\bfeta}^*),\bQ), \qquad \text{$\hat \bP-$a.s \quad in $L^2(0,T;L_2(U_0,\R	)).$ }
	\end{align}
			Now using classical ideas from \cite{Ben} (see Lemma 2.1 of \cite{DGHT}), we obtain from \eqref{g1} that
\begin{align}\label{G2}
\int_0^t(G(\hat{\bu}_{\ep},\hat{\bfeta}^*_{\ep})d\hat W_\ep,\bQ_{\ep}) \rightarrow \int_0^t( G(\hat \bu,\hat{\bfeta}^*)d\hat{W},\bQ),\,
\text{in probability in $L^2(0,T;\R)$}.
\end{align}			
Furthermore for some $C>0$ independent of $\ep$ we have the following bounds that follow from It\^{o}'s isometry:
\begin{align}
\hat \bE\int_0^T |\int_0^t (G(\hat{\bu}_{\ep},\hat{\bfeta}^*_{\ep})d\hat W_\ep(s),\bQ_{\ep})&|^2d t=\int_0^T\hat \bE\int_0^t\|(G(\hat{\bu}_{\ep},\hat{\bfeta}^*_{\ep}),\bQ_{\ep})\|^2_{L_2(U_0,\R)}d s d t \notag\\
& \le  T	\hat \bE\int_0^T\left( \|{\hat{\bfeta}^*_{\ep}}\|_{\bH_0^2(0,L)}^2+\|\hat{\bu}_{\ep}\|^2_{\bL^2(\sO)}\right) \|\bQ_\ep\|^2_{\bL^2} d s\label{G3}\\
&\le   C.\notag
\end{align}
Here we also used the a.s. bounds $\|\bq_\ep\|_{L^\infty(0,T;\bL^2(\sO))} \leq C(\delta)(\|\bq\|_{L^\infty(0,T;\bL^2(\sO))}+\|\bfpsi\|_{L^\infty(0,T;\bL^2(0,L))})$.	Combining \eqref{G2}, \eqref{G3} with the Vitali convergence theorem, we conclude the proof of Lemma \ref{conv_G}.
\end{proof}	
The rest of the convergence results follows as in \cite{TC23} and the only term requiring further explanation is the boundary integral $\int_0^t\int_\Gamma S_{\hat\bfeta^*_\ep}(\hat\bv_\ep^*\cdot\hat\bn_\ep^*)(\hat\bu_\ep\cdot\bq_\ep)$. Observe that, due to the embedding $H^{\frac14}(\Gamma)\hookrightarrow L^4(\Gamma)$, Theorem \ref{skorohod2} and \eqref{convq}, $\hat\bu_\ep\cdot\bq_\ep$ converges to $\hat\bu\cdot\bq$ in $L^2(0,T;L^2(\Gamma))$. Combining this with \eqref{vweak2} and \eqref{convrest1}$_5$, we obtain the convergence of $\int_0^t\int_\Gamma S_{\hat\bfeta^*_\ep}(\hat\bv_\ep^*\cdot\hat\bn_\ep^*)(\hat\bu_\ep\cdot\bq_\ep)$ to $\int_0^t\int_\Gamma S_{\hat\bfeta^*}(\hat\bv^*\cdot\hat\bn^*)(\hat\bu\cdot\bq)$ a.s., thus completing the proof of Theorem \ref{exist2}.
\end{proof}
{\s Notice that the weak formulation in Theorem~\ref{exist2} still contains $\hat{\bfeta}^*(t)$ in several terms.
		We will now show that in fact $\hat{\bfeta}^*(t)$ can be replaced by the stochastic process $\hat{\bfeta}(t)$ to obtain the desired weak formulation 
	until some strictly positive stopping time 
		$T^{\bfeta}$.
		
		\begin{lem} \label{StoppingTime}{\bf{(Almost surely positive stopping time.)}}
			Let the deterministic initial data ${\bfeta}_0$ satisfy the assumptions \eqref{etainitial}.
			Then, for any $\delta=(\delta_1,\delta_2)$ satisfying \eqref{delta},
			there exists an almost surely positive stopping time $T^{\bfeta}$, given by
			\begin{equation}\label{stoppingT}
				T^{\bfeta}:=T\wedge\inf\{t>0:\inf_{\sO}  J_{\hat\bfeta}(t)\leq \delta_1\} \wedge\inf\{t>0:\|\hat{\bfeta}(t)\|_{\bH^s(0,L)}\geq \frac1{\delta_2} \},
			\end{equation}
			such that
			\begin{align}\label{etasequal2}
				\hat	{\bfeta}^*(t)=\hat{\bfeta}(t) \quad \text{for } t<T^{\bfeta}.
			\end{align}
		\end{lem}
		
		\begin{proof}
			We write the stopping time as 
		$$	T^{\bfeta}=T\wedge\inf\{t>0:\inf_{\sO}  J_{\hat\bfeta}(t)\leq \delta_1\} \wedge\inf\{t>0:\|\hat{\bfeta}(t)\|_{\bH^s(0,L)}\geq \frac1{\delta_2} \}=:T\wedge T^{\bfeta}_1+T^{\bfeta}_2.
			$$
Observe that using the triangle inequality, for any $\delta_0>{\delta_2}$, we obtain for  $T^{\bfeta}_2$ that
			\begin{align*}
				\hat\bP[&T^{\bfeta}_2=0, \|{\bfeta}_0\|_{\bH^2(0,L)}<\frac1{\delta_0}] =\lim_{\epsilon\rightarrow 0^+}\hat\bP[T^{\bfeta}_2<\epsilon,\|{\bfeta}_0\|_{\bH^2(0,L)}<\frac1{\delta_0}]\\
				&\leq \limsup_{\epsilon \rightarrow 0^+}\hat\bP[\sup_{t\in[0,\epsilon)}\|\hat{\bfeta}(t)\|_{\bH^s(0,L)}>\frac1{\delta_2},\|{\bfeta}_0\|_{\bH^2(0,L)}<\frac1{\delta_0}]\\
				&\leq \limsup_{\epsilon \rightarrow 0^+}\hat\bP[\sup_{t\in[0,\epsilon)}\|\hat{\bfeta}(t)-{\bfeta}_0\|_{\bH^s(0,L)}>\frac1{\delta_2}-\frac1{\delta_0}] \\
				&\leq \frac1{(\frac1{\delta_2}-\frac1{\delta_0})}\limsup_{\epsilon \to 0}\hat\bE[\sup_{t\in[0,\epsilon)}\|\hat{\bfeta}(t)-{\bfeta}_0\|_{\bH^s(0,L)}]\\
				&\leq \frac1{(\frac1{\delta_2}-\frac1{\delta_0})}\limsup_{\epsilon\rightarrow 0}\hat\bE[\sup_{t\in[0,\epsilon)}\|\hat{\bfeta}(t)-{\bfeta}_0\|^{1-\frac{s}2}_{\bL^2(0,L)}\|\hat{\bfeta}(t)-{\bfeta}_0\|^{\frac{s}2}_{\bH^2(0,L)}]\\
				&\leq \frac1{(\frac1{\delta_2}-\frac1{\delta_0})}\limsup_{\epsilon\rightarrow 0}\hat\bE[\sup_{t\in[0,\epsilon)}\epsilon\|\hat\bv(t)\|^{1-\frac{s}{2}}_{\bL^2(0,L)}\|\hat{\bfeta}(t)-{\bfeta}_0\|^{\frac{s}{2}}_{\bH^2(0,L)}]\\
				&\leq \limsup_{\epsilon\rightarrow 0} \frac{\epsilon}{(\frac1{\delta_2}-\frac1{\delta_0})}\left( \hat\bE[\sup_{t\in[0,\epsilon)}\|\hat\bv(t)\|^{2}_{\bL^2(0,L)}]\right)^{\frac{2-s}4} \left( \hat\bE[\sup_{t\in(0,\epsilon)}\|\hat{\bfeta}(t)-{\bfeta}_0\|^{2}_{\bH^2(0,L)}]\right)^{\frac{s}4} \\
				&=0.
			\end{align*}
			Hence, by continuity from below, we infer that for any $\delta_2>0$, 
			\begin{align}
				\hat\bP[T^{\bfeta}_2=0, \|{\bfeta}_0\|_{\bH^2(0,L)}<\frac1{\delta_2}]=0.
			\end{align}
			
			We estimate $T^{\bfeta}_1$ similarly, by observing that for any $t\in[0,T]$, we have $\inf_{\sO}  J_{\hat\bfeta}(t)\geq \inf_{\sO} \hat J_0-\| J_{\hat\bfeta}(t)-J_0\|_{C(\sO)}$. Hence, for any $\delta_0>\delta_1$ we write
			\begin{align*}
				\hat\bP[T^{\bfeta}_1=0, \inf_{\sO}J_0>\delta_0] &\leq \limsup_{\epsilon \rightarrow 0^+}\hat\bP[\inf_{t\in[0,\epsilon)}\inf_{\hat\sO}  J_{\hat\bfeta}(t)<\delta_1,\inf_{\sO} \hat J_0>\delta_0]\\
				&\leq \limsup_{\epsilon \rightarrow 0^+}\hat\bP[\sup_{t\in[0,\epsilon)}\| J_{\hat\bfeta}(t)-J_0\|_{C(\sO)}>\delta_0-\delta_1] \\
				&\leq \frac1{(\delta_0-\delta_1)^2}\limsup_{\epsilon\rightarrow 0}\hat\bE[\sup_{t\in[0,\epsilon)}\| J_{\hat\bfeta}(t)-J_0\|^2_{C(\sO)}]\\
				&=0.
			\end{align*}
			Thus, for given $\delta_1>0$, 
			\begin{align}
				\hat\bP[T^{\bfeta}_1=0, \inf_{\sO}J_0>\delta_1]=0.
			\end{align}
		In conclusion we have,
			\begin{align}
				\hat\bP[T^{\bfeta}=0, \inf_{\sO}J_0>\delta_1, \|{\bfeta}_0\|_{\bH^2(0,L)}<\frac1{\delta_2}]=0.
			\end{align}
			
		\end{proof}
	}

	Finally, by combining Theorem \ref{exist2}, Lemma~\ref{StoppingTime}, and  \eqref{etasequal2}, we conclude the proof of our main result:	
			\begin{thm}\label{MainTheorem}{\bf{(Main result.)}}
			For any given $\delta=(\delta_1,\delta_2)$ satisfying \eqref{delta}, if the deterministic initial data  ${\bfeta}_0$ satisfies \eqref{etainitial}, then the stochastic processes $(\hat\bu,\hat{\bfeta},T^{\bfeta})$ along with the stochastic basis constructed in Theorem \ref{skorohod2}, 
			determine a martingale solution 
			in the sense of Definition \ref{def:martingale} of the stochastic FSI problem \eqref{u}-\eqref{ic}.\end{thm}


	\bibliography{references}
	
	\bibliographystyle{plain}
\end{document}